\documentclass[aap,MSNbibl,citesort,dvips]{arximspdf}
\usepackage{mathbh}
\usepackage{graphicx}

\doi{10.1214/10-AAP736}
\volume{21}
\issue{4}
\pubyear{2011}
\firstpage{1493}
\lastpage{1536}

\makeatletter

\newtheorem{theorem}{Theorem}[section]
\newtheorem{lemma}[theorem]{Lemma}
\newtheorem{corollary}[theorem]{Corollary}
\newtheorem{conjecture}[theorem]{Conjecture}

\newproclaim{ex}{Example}

\newcommand{\sm}{\setminus}
\newcommand{\E}{\mathbb{E}}
\newcommand{\N}{\mathbb{N}}
\newcommand{\Z}{\mathbb{Z}}
\newcommand{\bone}{\mathbh{1}}
\newcommand{\cA}{\mathcal{A}}
\newcommand{\cB}{\mathcal{B}}
\newcommand{\cC}{\mathcal{C}}
\newcommand{\cD}{\mathcal{D}}
\newcommand{\cE}{\mathcal{E}}
\newcommand{\cF}{\mathcal{F}}
\newcommand{\cG}{\mathcal{G}}
\newcommand{\cH}{\mathcal{H}}
\newcommand{\cM}{\mathcal{M}}
\newcommand{\cP}{\mathcal{P}}
\newcommand{\cT}{\mathcal{T}}
\newcommand{\cZ}{\mathcal{Z}}
\newcommand{\eps}{{\varepsilon}}

\makeatother

\begin{document}
\begin{frontmatter}

\title{Order-invariant measures on causal sets}
\runtitle{Order-invariant measures on causal sets}

\begin{aug}
\author[A]{\fnms{Graham} \snm{Brightwell}\ead[label=e1]{g.r.brightwell@lse.ac.uk}}
and
\author[A]{\fnms{Malwina} \snm{Luczak}\corref{}\ead[label=e2]{m.j.luczak@lse.ac.uk}\thanksref{t1}}
\runauthor{G. Brightwell and M. Luczak}
\affiliation{London School of Economics and Political Science}
\thankstext{t1}{Supported in part by a grant
from STICERD.}
\address[A]{Department of Mathematics\\
London School of Economics\\
\quad and Political Science\\
London WC2A 2AE\\
United Kingdom\\
\printead{e1}\\
\hphantom{E-mail: }\printead*{e2}} %adresu isvedimo komanda gale!
\end{aug}

% HISTORY:
\received{\smonth{1} \syear{2009}}
\revised{\smonth{4} \syear{2010}}

% ABSTRACT
%
\begin{abstract}
A causal set is a partially ordered set on a countably infinite ground-set
such that each element is above finitely many others. A~natural extension
of a causal set is an enumeration of its elements which respects the order.

We bring together two different classes of random processes. In one
class, we are given a fixed causal set, and we consider random natural
extensions of this causal set: we think of the random enumeration as being
generated one point at a time. In the other class of processes, we
generate a random causal set, working from the bottom up,
adding one new maximal element at each stage.

Processes of both types can exhibit a property called order-inva\-riance: if
we stop the process after some fixed number of steps, then, conditioned on
the structure of the causal set, every possible order of generation of its
elements is equally likely.

We develop a framework for the study of order-invariance which includes
both types of example: order-invariance is then a property of probability
measures on a certain space. Our main result is a description of the
extremal order-invariant measures.
\end{abstract}

% KEYWORDS
%

\end{frontmatter}

%s1 ###
\section{Introduction} \label{sec:intro}

This work is intended as a common generalization of two different strands
of research: a proposal from physicists for a mathematical model of
space--time as a discrete poset, and a notion of a ``random linear
extension'' of an infinite partially ordered set. One of our aims is to
show that these two lines of research are intimately connected.

The objects we study are \textit{causal sets}, which are countably infinite
partially ordered sets $P=(Z,<)$ such that every element is above only
finitely many others. A \textit{natural extension} of a causal set is a
bijection from $\N$ to $Z$ whose inverse is order-preserving; that is,
it is
an enumeration of $Z$ that respects the ordering $<$.

We consider random processes that generate a causal set one element at a
time, starting with the empty poset, and at each stage adding one new
maximal element, keeping track of the order in which the elements are
generated. Such a process is called a \textit{growth process}. The infinite
poset $P$ generated by a growth process is always a causal set, and the
order in which the elements are generated is a natural extension of $P$.

We will postpone most of the formal definitions for a while, although we
will introduce some notation that will be consistent with that\vadjust{\goodbreak} used in the
bulk of the paper. Our main purpose in this section is to motivate the
ideas of the paper by examining some examples. Before that, we need a
little terminology.

A \textit{(labeled) poset} $P$ is a pair $(Z,<)$, where $Z$ is a set (for
us, $Z$ will always be countable), and $<$ is a partial order on $Z$,
that is,
a transitive irreflexive relation on~$Z$.
An order $<$ on $Z$ is a \textit{total order} or \textit{linear order} if
each pair $\{a,b\}$ of distinct elements of $Z$ is comparable ($a<b$ or
$b<a$).

A \textit{down-set} in $P$ is a subset $Y \subseteq Z$ such that, if $a\in Y$
and $b<a$, then $b \in Y$. An \textit{up-set} is the complement of a
down-set: a set $U \subseteq Z$ such that $b\in U$ and $a>b$ implies
$a\in U$.

A pair $(x,y)$ of elements of $Z$ is a \textit{covering pair} if $x<y$, and
there is no $z\in Z$ with $x<z<y$. We also say that $x$ is \textit{covered}
by $y$, or that $y$ \textit{covers} $x$.

If $P=(Z,<)$ is a poset, and $Y \subseteq Z$, then $<_Y$ denotes the
restriction of the partial order to $Y$, and $P_Y=(Y,<_Y)$. For
$W \subset Z$, we also write $P\sm W$ to mean $P_{Z\sm W}$.

For $P = (Z,<)$ a poset on any ground-set $Z$, a \textit{linear extension} of
$P$ is a~total order $\prec$ on $Z$ such that, whenever $x<y$, we also have
$x \prec y$. In the case where $Z$ is finite, the set of linear extensions
is also finite.

We will often be considering posets on the set $\N$, or on one of the sets
$[k]=\{1,\ldots, k\}$, for $k \in\N$, which come equipped with a
``standard'' linear order. In these cases, a \textit{suborder} of $\N$ or
$[k]$ will be a partial order on that ground-set (typically denoted
$<^\N$
or $<^{[k]}$) with the standard order as a linear extension, that is, if
$<^\N$ is a suborder of $\N$ and $i<^\N j$, then $i$ is below $j$ in the
standard order on $\N$.

In the case where the ground-set $Z$ of $P$ is countably infinite, the
natural extensions of $P$ correspond to the linear extensions $\prec$ with
the order-type of the natural numbers: specifically, given a natural
extension of $P$, which is a bijection $\lambda\dvtx \N\to Z$ whose
inverse is
order-preserving, we obtain a linear extension $\prec$ of $P$ by setting
$\lambda(i) \prec\lambda(j)$ whenever $i < j$ in the standard order on
$\N$.

\begin{ex}\label{ex1}
Figure~\ref{fig:ladder} below shows the Hasse diagram of a labeled causal
set $P=(Z,<)$, where $Z = \{ a_1, a_2, \ldots\}$, and $a_j>a_i$ if $j>i+1$.
(Later, we will require that the $a_i$ are distinct real numbers in
$[0,1]$, but the order $<$ imposed on the $a_i$ by $P$ has no relation to
the order of $[0,1]$.)

%f1 ###
\begin{figure}

\includegraphics{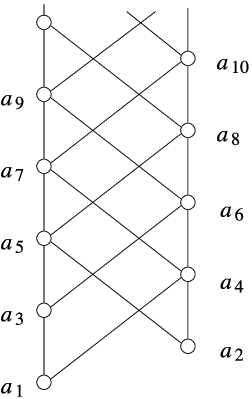}

\caption{The causal set $P=(Z,<)$.} \label{fig:ladder}
\end{figure}

The natural extensions of this poset $P$ are the bijections
$\lambda\dvtx \N\to Z$ such that, for $i<j$, $a_i\not> a_j$. Equivalently,
we require that $\{ \lambda(1), \ldots, \lambda(k) \}$ is a~down-set in
$P$, for each $k$.

We are interested in a particular probability measure $\mu$ on the set
$L(P)$ of natural extensions $\lambda$ of $P$, which has properties one
would associate with a~``uniform'' probability measure. The
$\sigma$-field of measurable sets is generated by events of the form
\[
E(a_{i_1}a_{i_2}\cdots a_{i_k}) = \{ \lambda\dvtx \lambda(j) = a_{i_j}
\mbox{ for } j=1, \ldots, k\},
\]
the set of natural extensions with ``initial segment''
$a_{i_1}a_{i_2}\cdots a_{i_k}$, for $k \in\N$ and the $i_j$ distinct
elements of $\N$. We call $a_{i_1}a_{i_2}\cdots a_{i_k}$ an
\textit{ordered stem} if $\{ a_{i_1}, \ldots, a_{i_j} \}$ is a down-set in
$P$, for $j=1, \ldots, k$: in other words if there is a~natural extension
of $P$ with this initial segment.

We describe the measure $\mu$ via a random process for generating the
sequence $\lambda(1), \lambda(2), \ldots$ sequentially. Given the set
$X_k = \{ \lambda(1), \lambda(2), \ldots, \lambda(k) \}$, the element
$\lambda(k+1)$ has to be one of the minimal elements of $P\sm X_k$, and
there are at most two of these. The random process we are interested in
is the one defined by the following rules:
\begin{itemize}
\item  if there is only one minimal element $a_k$ of $P\sm X_k$, take
$\lambda(k+1) = a_k$ with probability~1;
\item  if there are two minimal elements
$a_{k+1}$ and $a_{k+2}$ of $P\sm X_k$, set $\lambda(k+1) = a_{k+1}$ with
probability $\phi= \frac{1}{2}(\sqrt5 -1) = 0.618\ldots$ and
$\lambda(k+1) = a_{k+2}$ with probability $1-\phi$.
\end{itemize}
It is easy to see that the function $\lambda$ generated by these rules is
always a~natural extension of~$P$.

We have described this as a process generating a random natural extension,
but we can also think of it as a growth process, growing a causal set by
adding one new maximal element at each step: the process always generates
the same infinite causal set $P$, but the order in which the elements are
generated is random.

We now calculate
%
%e1 ###
\begin{equation} \label{eq:zero}
\mu(E(a_1a_2)) = \phi^2 = 1-\phi= \mu(E(a_2a_1)).
\end{equation}
Indeed, we choose $\lambda(1) = a_1$ with probability $\phi$; having done
so, we choose $\lambda(2) = a_2$ with probability $\phi$. On the other
hand, we choose $\lambda(1) = a_2$ with probability $1-\phi$; having
done so, $a_1$ is the only minimal element of $P \sm\{a_2\}$, so we
choose $\lambda(2) = a_1$ with probability~1.

Moreover, we claim that, whenever $a_{i_1}a_{i_2} \cdots a_{i_k}$ and
$a_{\ell_1}a_{\ell_2}\cdots a_{\ell_k}$ are two ordered stems with
$\{a_{i_1},\ldots,a_{i_k}\} = \{a_{\ell_1},\ldots,a_{\ell_k}\}$,
we have
%
%e2 ###
\begin{equation} \label{eq:one}
\mu(E(a_{i_1}a_{i_2} \cdots a_{i_k})) =
\mu(E(a_{\ell_1}a_{\ell_2}\cdots a_{\ell_k})).
\end{equation}
If the two orders $a_{i_1}a_{i_2}\cdots a_{i_k}$ and
$a_{\ell_1}a_{\ell_2}\cdots a_{\ell_k}$ differ only by an exchange of
adjacent elements---necessarily $a_r$ and $a_{r+1}$ for some $r$---then
(\ref{eq:one}) follows by essentially the same calculation as in
(\ref{eq:zero}): the two probabilities
$\mu(E(a_{i_1}a_{i_2} \cdots a_{i_k}))$ and
$\mu(E(a_{\ell_1}a_{\ell_2}\cdots a_{\ell_k}))$ are products of terms
which are the same except that one has two terms equal to $\phi$ and the
other has one term equal to $1-\phi$ and another equal to~1.
To see (\ref{eq:one}) in general, it suffices to show that we can step from
$a_{i_1}a_{i_2}\cdots a_{i_k}$ to $a_{\ell_1}a_{\ell_2}\cdots a_{\ell_k}$
by a sequence of exchanges of adjacent elements, staying within the set
of ordered stems. This is a standard fact about the set of linear extensions
of any finite poset: to see it in this case, start with the order
$a_{i_1}a_{i_2}\cdots a_{i_k}$, and move each $a_{\ell_j}$ in turn down
until it reaches position~$j$.

The property in (\ref{eq:one}) is called \textit{order-invariance}. If we
consider instead a~\textit{finite} poset $P=(Y,<)$, then the uniform
probability measure $\nu^P$ on the set of linear extensions of $P$
satisfies order-invariance. Indeed, another way of obtaining the measure
$\mu$ in our example is to consider the sets $Z_n = \{a_1, \ldots,
a_n\}$,
the finite posets $P_n = P_{Z_n}$, and the uniform measures~$\nu
^{P_n}$ on
their sets of linear extensions, for each $n$. It can be shown that
\[
\nu^{P_n}(E(a_{i_1}a_{i_2} \cdots a_{i_k})) \to
\mu(E(a_{i_1}a_{i_2} \cdots a_{i_k}))
\]
as $n \to\infty$, for each ordered initial segment
$a_{i_1}a_{i_2} \cdots a_{i_k}$.
\end{ex}

Our second example is apparently of a very different nature. We consider
a family of probability measures on the set of causal sets with ground-set
$\N$---that is, models of random causal sets---and explain how these
measures also satisfy an order-invariance property.

\begin{ex}\label{ex2}
A \textit{random graph order} $P = (\N,\prec)$, with parameter $p \in(0,1)$,
is defined on the set $\N$ as follows. We take a random graph on $\N$---for each pair $(i,j)$ of elements of $\N$, we put an edge between $i$ and
$j$ with probability $p$, all choices made independently. Then we define
the random order $\prec$ from the random graph by declaring that
$i \prec j$ if there is an increasing sequence
$i = i_1,i_2,\ldots,i_m = j$ of natural numbers such that
$i_\ell i_{\ell+1}$ is an edge for each $\ell=1,\ldots,m-1$.

Equivalently, we could define the random graph order with parameter $p$
via a growth process, adding a new maximal element at each stage. Given
the restriction $P_{[k]}$ to the set $[k]$, at the next step of the
process, a random subset $\Sigma$ of $[k]$ is chosen, with each element
taken into $\Sigma$ independently with probability $p$. Then $k+1$ is
placed above the elements of $\Sigma$, and the transitive closure is taken---so if $j$ is in $\Sigma$ and $i \preceq j$ in $P_{[k]}$, then $i$ is
placed below $k+1$ in $P_{[k+1]}$.

This is a model of random posets---there are versions with the ground-set
being a finite set $[n]$, or $\Z$---with a number of interesting
features, and it also has the advantage that it is relatively easy to
analyze. Accordingly, random graph orders have attracted a fair degree of
attention in the combinatorics literature; see, for
instance, \cite{AF,ABBJ,BB,PT}.

Fix some $k \in\N$, and some suborder $<^{[k]}$ of $[k]$. We claim that
the probability that the order $\prec_{[k]}$ on $[k]$ is equal to
$<^{[k]}$ is given by
%
%e3 ###
\begin{equation} \label{eq:c-b}
p^{c(<^{[k]})}(1-p)^{b(<^{[k]})},
\end{equation}
where $c(<^{[k]})$ is the number of covering pairs of $([k],<^{[k]})$, and
$b(<^{[k]})$ is the number of incomparable pairs.

To see this, note that, if $i$ is covered by $j$ in $<^{[k]}$, then in
order for $\prec_{[k]}$ to equal $<^{[k]}$, it is necessary for $ij$
to be
an edge of the random graph. Also, if $i$ and $j$ are incomparable in
$<^{[k]}$, then it is necessary for $ij$ to be a non-edge. Conversely, if
$i <^{[k]} j$, but $i$ is not covered by $j$, then there is some sequence\vspace*{1pt}
$i =i_1i_2\cdots i_m=j$ of elements of $[k]$ such that $i_\ell$ is covered
by $i_{\ell+1}$ in $<^{[k]}$, for $\ell=1, \ldots, m-1$. Provided that
each edge $i_\ell i_{\ell+1}$ is in the random graph, we will have
$i \prec j$ whether or not the edge $ij$ is in the random graph. Thus,
$\prec_{[k]}$ is equal to $<^{[k]}$ if and only if all the covering pairs
of $<^{[k]}$ span edges in the random graph, and all the incomparable
pairs do not.

The key point for our purposes is that the expression (\ref{eq:c-b}) is
an isomorphism-invariant of the poset $<^{[k]}$, and so isomorphic posets
have equal probabilities of arising as $\prec_{[k]}$. We again call this
property \textit{order-invariance}. An interpretation is that, if we stop
the process when there are $k$ elements, and look at the structure of the
poset, but not at the numbering of the elements, then, conditioned on this
information, each linear extension of the poset is equally likely to have
been the order in which the elements were generated.
\end{ex}

Growth processes, of a type similar to those in Example~\ref{ex2}, were
investigated by Rideout and Sorkin \cite{RS}, who view them as possible
discrete models for the space--time universe. The idea is that the
elements of the (random) causal set form the (discrete) set of points in
the space--time universe, and the partial order $\prec$ is interpreted as
``is in the past light-cone of.''

The order in which the elements of the causal set are generated is not
deemed to have any physical meaning, so it should not be possible to
extract information about this order from the causal set at any stage.
Rideout and Sorkin thus viewed growth processes as being Markov chains on
the set of finite \textit{unlabeled} causal sets, where each transition adds
a new maximal element. They studied such processes with the property
that, conditional on the causal set at some stage $k$ being equal to some
unlabeled $k$-element poset~$P$, each linear extension of $P$ is equally
likely to have been the order in which the elements were generated. They
called this property ``general covariance.'' Alternatively, we can view
the Rideout--Sorkin processes as generating an order on the ground-set
$\N$, as in Example~\ref{ex2}; then the property of general covariance translates
to the property of order-invariance, as described in Example~\ref{ex2}.

In \cite{RS}, Rideout and Sorkin characterized all growth processes
satisfying general covariance as well as another condition called
\textit{Bell causality}, and also a ``connectedness'' condition that prevents
the model breaking up as a sequence of models of posets stacked on top of
one another. The models satisfying all three conditions are called
\textit{classical sequential growth models} or \textit{csg models}; these were
studied further in \cite{RS2,Georgiou,BG}. Random graph orders, as in
Example~\ref{ex2}, are the prime examples of csg models. A general csg model can
be described in similar terms to our description of a random graph order
as a growth model; the particular csg model is specified by a sequence of
real parameters $t_n$ representing the relative probability of choosing
the random set $\Sigma$ to be equal to a given set $S$ of size~$n$.

Brightwell and Georgiou \cite{BG} determined that the large-scale
structure of any csg model is that of a semiorder, and in particular is
quite unlike the observed space--time structure of the universe.

Varadarajan and Rideout \cite{VR} and Dowker and Surya \cite{DS} describe
the models that can arise if the connectedness condition is dropped. Here
there is a fascinating extra layer of complexity: the causal sets arising
are all obtained by stacking ``csg models'' on top of one another, and
the sizes of ``later'' components may depend on the detailed structure of
``earlier'' ones if these are finite.

The underlying reason that csg models cannot produce causal sets that
resemble the observed universe seems to lie with the condition of Bell
causality: it is possible to show that any process producing causal
sets of
the desired type (essentially, those induced on a discrete set of points
arising from a~Poisson process on a Lorentzian manifold) will not satisfy
this condition.

Our aim in this paper is, effectively, to study the class of growth
processes satisfying general covariance: this class is vastly richer than
the class of csg models. For instance, if we drop the labels $a_i$ from
the causal set in Example~\ref{ex1}, and consider the growth process that we
described there as being a process on unlabeled posets, then the property
of order-invariance again translates to general covariance.

Dealing with unlabeled combinatorial structures is often awkward; in
cases similar to Example~\ref{ex1}, it is also very unnatural. So we will deal
with labeled causal sets from now on, and we want to express
order-invariance in terms of notation similar to that used in Example~\ref{ex1}.

We are thus faced with the problem of how to incorporate random graph
orders (and other csg models) into our setting. The numbering of the
elements that we used in Example~\ref{ex2} specifies the order of generation of
the elements, and so these numbers cannot serve as labels in the same
sense as the $a_i$ are used to label the elements in Example~\ref{ex1}.

It is useful at this point to introduce another family of examples, in
some ways trivial but in other ways far from it.

\begin{ex}\label{ex3}
We consider growth processes where the causal set generated is a.s.\ an
antichain (i.e., no two elements are comparable). This is the case if we
take a random graph order with $p=0$: we certainly do want to include some
such growth processes within our framework.

If we require our causal sets to be labeled, then a growth process which
a.s.\ generates an antichain is nothing more than a sequence of random
variables: the labels of the elements, in the order they are introduced.

Order-invariance requires that, if we condition on the set of the first
$k$ labels, for any $k$, then each of the $k!$ orderings of these labels
is equally likely. This is exactly the requirement that the sequence of
labels be \textit{exchangeable}.

One way to generate a sequence of exchangeable random labels is to take
any probability distribution $\tau$ on any set $X$ of potential labels,
and let the labels be an i.i.d. sequence of random elements of $X$ with
probability measure $\tau$. We will want our labels to be a.s.\ distinct,
so we need the probability measure $\tau$ to be atomless.

The Hewitt--Savage theorem \cite{HS} states that every sequence of
exchangeable random variables is a mixture of sequences of the type
described above (i.e., there is a probability measure $\rho$ on some space
of probability measures on a set $X$: one measure $\tau$ is chosen
according to $\rho$, and then an i.i.d. sequence of random elements of
$X$ is
generated according to $\tau$).

For instance, we can take $X$ to be the interval $[0,1]$, equipped with
its usual Borel $\sigma$-field and Lebesgue probability measure, and
$\tau$ to be the uniform probability measure on $X$. Our growth process
then operates as follows: at each stage, we introduce a new element,
labeled with a uniformly random element of $[0,1]$, chosen independently
of all other labels, and we make the new element incomparable with all
existing elements. This is indeed order-invariant: if we condition on the
state of the process after $k$ steps---an antichain labeled with a set
of $k$ numbers from $[0,1]$, a.s.\ distinct---then each of the $k!$
orders of generation is equally likely.
\end{ex}

Formally, we will handle random graph orders in exactly the same way as in
the example above: our growth process will proceed by taking a new
element, assigning it a uniformly random label from $[0,1]$, independent
of any other labels and of the structure of the existing poset, and then
placing the new element above some of the existing elements as described
in Example~\ref{ex2}. Such a growth process will be order-invariant.

In general, it is convenient to work only with causal sets labeled by
elements from a specific set, and we shall choose the interval $[0,1]$,
which comes equipped with its standard (compact) topology, and the Borel
$\sigma$-field~$\cB$ generated by the topology.

One generally applicable way of specifying the outcome of a growth process
is by giving an \textit{infinite} string of (labels of) elements, listed as
$x_1x_2\cdots$ in the order of their generation, together with a suborder
$<^\N$ of the index set $\N$ with its standard order: $i<^\N j$ if and
only if $x_i < x_j$ in the causal set $P=(X,<)$ generated by the process.

Growth processes thus correspond to probability measures on the set
$\Omega$ of pairs
\[
(x_1x_2 \cdots, <^\N),
\]
where the $x_i$ are elements of $[0,1]$ and $<^\N$ is a suborder of
$\N$.
We will proceed by taking $\Omega$ as the outcome space, with the
appropriate $\sigma$-field $\cF$, and considering probability
measures on
$(\Omega,\cF)$. We will set up the notation carefully in
Section~\ref{sec:csp}, introducing the notion of a
\textit{causal set process} or \textit{causet process}, which is effectively the
same as a growth process, but where the states are formally pairs
$(x_1\cdots x_k, <^{[k]})$, where the poset $<^{[k]}$ is on the index set
$[k]$, rather than on the set $X_k = \{ x_1, \ldots, x_k\}$. We give a
formal definition of order-invariance, as a property of probability
measures on $(\Omega, \cF)$, in Section~\ref{sec:oip}.

We emphasize that we will build one space $(\Omega,\cF)$ to accommodate
\textit{all} causet processes, subject only to the fairly arbitrary
restriction that the set of potential labels of elements is $[0,1]$. We
will then study the space of all order-invariant measures, which we will
define as probability measures on $(\Omega,\cF)$ satisfying a certain
condition. This space of order-invariant measures has some good
properties; for instance, it is a convex subset of the set of all
probability measures on $(\Omega,\cF)$, and we shall show in
Section~\ref{sec:invariant} that it is closed in the topology of weak
convergence.

In order to make a systematic study of order-invariant measures, we shall
focus on the \textit{extremal} order-invariant measures: those that cannot
be written as a convex combination of two others.

An order-invariant measure that almost surely produces one fixed
(labeled) causal set $P=(Z,<)$, as in Example~\ref{ex1}, will be called an
\textit{order-invariant measure on $P$}. The process in Example~\ref{ex1} is in fact
the only order-invariant measure on the poset $P$ of
Figure~\ref{fig:ladder}, and it is extremal. We shall see an example
later of a causal set with infinitely many extremal order-invariant
measures on it.

On the other hand, it follows from the analysis in Example~\ref{ex3} that
a~%
labeled antichain $(Z,<)$ admits no order-invariant measures. Indeed
we saw that, if an order-invariant measure generates an antichain a.s.,
then there is a~probability measure $\rho$ on the space of probability
measures on $([0,1],\cB)$, such that the sequence of labels is generated
by first choosing a probability measure $\tau$ according to $\rho$, then
taking an i.i.d. sequence of random variables with distribution~$\tau$.
Now, if $x \in[0,1]$, and $x$ occurs as a label with positive
probability, then $\rho( \tau(\{x\}) > 0) > 0$, and in that case $x$
occurs as a label infinitely often with positive probability. So
such a process cannot generate each label in $Z$ exactly once.

There is however an abundance of extremal order-invariant measures that
are measures on some fixed causal set. Also, there are extremal
order-invariant measures a.s.\ giving rise to an antichain: it follows
from the discussion in Example~\ref{ex3}---in particular, from the
Hewitt--Savage theorem~\cite{HS}---that these are effectively the same as
i.i.d. sequences of random elements of $[0,1]$.

Our main result is Theorem~\ref{thm:extremal}, showing that all extremal
order-invariant measures on $(\Omega, \cF)$ are, in a sense to be made
precise later, a combination of extremal order-invariant measures of these
two types.

Sections~\ref{sec:csne}--\ref{sec:oip} are devoted to defining notation
and terminology, setting up the spaces we are studying, and giving precise
definitions. We also establish some useful properties of order-invariant
measures in Section~\ref{sec:oip}. In Section~\ref{sec:tsc}, we give
details of how examples such as the ones in this section fit into the
general framework. In Section~\ref{sec:invariant}, we show that the set
of order-invariant measures is the set of measures invariant under a
certain family of permutations on $\Omega$, and we derive as a consequence
that the set of order-invariant measures is closed in the topology of weak
convergence. In Section~\ref{sec:eoim}, we give a number of conditions
equivalent to extremality of an order-invariant measure, and also show that
every order-invariant measure is a mixture of extremal ones. Finally,
in Section~\ref{sec:deoim}, we state, discuss and prove
Theorem~\ref{thm:extremal}.

Our results do not provide a \textit{classification} of extremal order-invariant
measures: this would necessarily involve a classification of extremal
order-invariant measures on fixed causal sets, which seems likely to be
prohibitively difficult. However, some partial results in this
direction are
given in the authors' companion paper~\cite{BL2}, where order-invariant
measures on fixed causal sets are studied in depth.

For now, we just point out some more connections to existing literature.
Some years ago, the first author~\cite{Bri1,Bri2} studied random linear
extensions of locally finite posets. The main theorem of~\cite{Bri1},
interpreted in the present context, is as follows. If a causal set $P$
has the property that, for some fixed~$t$, every element is incomparable
with at most $t$ others, then there is a unique order-invariant measure on
$P$. For instance, this applies to the causal set in Example~\ref{ex1}. More
details can be found in~\cite{BL2}.

The specific case where $P$ is the two-dimensional grid $(\N\times\N,<)$
has attracted considerable attention, as it is connected with the
representation theory of the infinite symmetric group, and with harmonic
functions on the Young lattice (which is the lattice of down-sets of $P$).
A good account of this theory appears in Kerov~\cite{Kerov}, where a
somewhat more general theory is also developed. Our concerns in this
paper are different, but the two theories have various points of contact.

The family of natural extensions of a fixed causal set $P$ can also be
viewed as the set of configurations of a (1-dimensional) spin system, and
order-invariant measures can then be interpreted as \textit{Gibbs measures},
so that some of the general results discussed in, for instance,
Bovier~\cite{Bovier} or Georgii~\cite{Georgii} apply. In fact, as we shall
see later, some of the results in~\cite{Georgii} apply to order-invariant
measures in general.

%s2 ###
\section{Causal sets and natural extensions} \label{sec:csne}

For a poset $P=(Z,<)$ and an element $x\in Z$, set
$D(x)=\{ y \in Z\dvtx y<x\}$, the set of elements below $x$. We also set
$U(x) = \{ y \in Z \dvtx y > x\}$ and $I(x)$ to be the set of elements
incomparable with $x$. Thus, $\{ D(x), I(x), U(x) \}$ is a partition of
$Z \sm\{x\}$. A causal set is a poset in which $D(x)$ is finite for
all~$x$.

Recall that a \textit{natural extension} of a causal set $P=(Z,<)$ is a
bijection $\lambda$ from $\N$ to $Z$ such that $\lambda^{-1}$ is
order-preserving: that is, if $\lambda(i) < \lambda(j)$, then $i<j$.
It is
often convenient to write natural extensions as $x_1x_2 \cdots$, meaning
that $\lambda(i) =x_i$. In this notation, an \textit{initial segment} of
$\lambda$ is an initial substring $x_1x_2 \cdots x_k$, for some $k\in
\N$.

A natural extension $\lambda$ of a countably infinite poset $P=(Z,<)$
gives rise to a linear extension $\prec$ of $P$ by setting $x \prec y$
whenever $\lambda^{-1}(x) < \lambda^{-1}(y)$. The linear extensions
arising in this way are those with the order-type of $\N$.

Similarly, if $P=(Z,<)$ is a~finite poset, with $|Z|=k$, we can think
of a~linear extension as a bijection $\lambda\dvtx [k] \to Z$ such that
$\lambda^{-1}$ is order-preserving, that is, if $\lambda(i) < \lambda(j)$,
then $i<j$ in~$[k]$. We shall sometimes write a linear extension of a
finite poset $P$ as $x_1\cdots x_k$, meaning that $\lambda(i) = x_i$ for
$i=1,\ldots,k$. For finite partial orders, we shall use these various
equivalent notions of linear extension interchangeably.

A \textit{stem} in a causal set is a finite down-set (this term is less
standard: it has been used in some physics papers). An \textit{ordered stem}
of a causal set $P=(Z,<)$ is a finite string $x_1\cdots x_k$ such that
$X=\{x_1, \ldots, x_k\}$ is a down-set in $P$, and $x_1\cdots x_k$ is a
linear extension of $P_X$. In other words, ordered stems are exactly the
strings that can arise as an initial segment of a natural extension
of~$P$.

For a countable poset $P=(Z,<)$, let $L(P)$ denote the set of natural
extensions of $P$. Also, let $L'(P)$ denote the set of injections
$\lambda$
from $\N$ to~$Z$ such that, for each $i$,
$D(\lambda(i)) \subseteq\{\lambda(1), \ldots, \lambda(i-1)\}$. In
general, elements of $L'(P)$ need not be bijections from $\N$ to $Z$: they
may be invertible maps from~$\N$ onto a proper subset of $Z$, which will
necessarily be an infinite down-set in $P$. Those elements of $L'(P)$ that
are bijections from $\N$ to $Z$ are exactly the natural extensions of $P$.

A countable poset has a natural extension if and only if every element is
above finitely many elements, that is, if and only if it is a causal
set. If
$P$ has no element $x$ with $I(x)$ infinite, then all linear extensions of
$P$ correspond to natural extensions, and $L(P)=L'(P)$. However, if there
is an element $x$ of $P$ with $I(x)$ infinite, then there is (a)~a linear
extension of $P$ that does not have the order-type of $\N$ and (b)~an
element of $L'(P)$ whose image is the proper subset $I(x) \cup D(x)$
of~$P$.

%s3 ###
\section{Causal set processes} \label{sec:csp}

A \textit{causal set process} or \textit{causet process} is a discre\-te-time
Markov chain on an underlying probability space $(\Omega, \cF, \mu
)$, that
we shall specify shortly. The elements of the state space $\cE$ of the
Markov chain are ordered pairs $(x_1\cdots x_k, <^{[k]})$, where
$x_1\cdots x_k$ is a string of elements from $[0,1]$, and $<^{[k]}$ is a
suborder of $[k]$. The only permitted transitions of the chain are
\textit{one-point extensions}, from a pair $(x_1\cdots x_k, <^{[k]})$ to a
pair $(x_1\cdots x_kx_{k+1}, <^{[k + 1]})$,\vspace*{1pt} where $x_{k+1}$ is an
element of $[0,1]$, and $<^{[k + 1]}$ is obtained from $<^{[k]}$ by
adding $k + 1$ as a maximal element. A transition from the state
$(x_1\cdots x_k, <^{[k]})$ is thus specified by the element $x_{k+1}$ of
$[0,1]$ to be appended to the string, and the set $D(k + 1)$, a down-set
in $([k],<^{[k]})$.

From each state $(x_1\cdots x_k, <^{[k]})$, we can derive a partial order
\mbox{$P_k=(X_k,<)$}, with ground-set $X_k=\{ x_1, \ldots, x_k\}$, and $x_i < x_j$\vspace*{1pt}
if and only if $i<^{[k]}j$. We always interpret $<^{[k]}$ as giving a
partial order $<$ on $X_k$ in this way. The condition that $<^{[k]}$ is
a suborder of $[k]$ then translates to the condition that the linear order
$x_1\cdots x_k$ is a linear extension of $P_k$; indeed the states of the
causet process are in 1--1 correspondence with the set of pairs
$(P_k,x_1\cdots x_k)$, where $P_k$ is a poset on $\{x_1,\ldots,x_k\}$ and
$x_1\cdots x_k$ is a linear extension of~$P_k$. In this interpretation,
as in Section~\ref{sec:intro}, a transition adds a new maximal element,
drawn from $[0,1]$, to $P_k$.

For fixed $k\in\N$, let $\cE^{[k]}$ be the set of states
$(x_1\cdots x_k, <^{[k]}) \in\cE$, that is, those with $k$ elements.
So all
permitted transitions go from $\cE^{[k]}$ to $\cE^{[k + 1]}$, for some
$k$.

We shall declare our underlying outcome space $\Omega$ and $\sigma$-field
$\cF$ to be the simplest structure supporting all causet processes. The
outcome space $\Omega$ can thus be taken to consist of all possible
sequences of states, starting from the empty string. Now, each
$\omega\in\Omega$ can be identified with a pair ($x_1x_2\cdots,<^\N$),
where $x_1x_2\cdots$ is an infinite sequence of elements of $[0,1]$, and
$<^\N$ is a suborder of $\N$. It is convenient for us to \textit{define}
$\Omega$ as the set of all such pairs $\omega= (x_1x_2\cdots, <^\N)$.

We define the projections $\pi_k\dvtx \Omega\to\cE^{[k]}$ by
\[
\pi_k (x_1x_2\cdots, <^\N) = \bigl(x_1\cdots x_k, <^\N_{[k]}\bigr)
\]
(in line with our general notation, $<^\N_{[k]}$ denotes the restriction\vspace*{1pt}
of the order $<\N$ on the ground-set $\N$ to the subset $[k]$). In other
words, $\pi_k$ is the ``restriction'' of $\omega= (x_1x_2\cdots,<^\N
)$ to
its first $k$ entries. Thus, the sequence~%
$\pi_0(\omega)$, $\pi_1(\omega),\ldots$ is the sequence of states
corresponding to the outcome $\omega$. The map~$\pi_k$ is then seen as
the natural projection on to the $k$th state (and so in this case on to
all the first $k$ states) in the sequence.

Given an element $\omega= (x_1x_2\cdots, <^\N)$ of $\Omega$, we can
derive a countably infinite subset $X = \{x_1, x_2, \ldots\}$ of $[0,1]$,
together with a poset $P=(X,<)$ on~$X$, where $x_i < x_j$ if and
only if $i<^\N j$, and a natural extension $x_1x_2 \cdots$ of~$P$.
Conversely, such a triple $(X,<_P,x_1x_2 \cdots)$ determines
$\omega\in\Omega$ uniquely. The sequence $(P_k)$ of finite posets can
be obtained from $P$ by setting \mbox{$P_k=P_{X_k}$}, the restriction of $P$ to
$X_k = \{ x_1, \ldots, x_k\}$, for each $k$.

We need some notation for functions on $\Omega$, that is, random elements
on our probability space; where possible, for an object denoted by a
Roman letter, we will use the Greek version of the letter to denote the
corresponding random element. Thus, we will denote by $\xi_k$ the random
$k$th coordinate, that is, the element in $[0,1]$ with
$\xi_k(\omega) = \xi_k(x_1x_2\cdots, <^\N) = x_k$. We shall use~$\Xi_k$
to denote the random set $\{ \xi_1, \ldots, \xi_k\}$, and $\Xi$ to denote
the random set $\{ \xi_1, \xi_2, \ldots\}$. We use $\Delta_k$ to
denote the random element taking values in the set of subsets of $[k]$
with $\Delta_k(\omega) = D(k)$, the down-set of elements below~$k$ in
$<^\N$. Finally, we will use $\prec^\N$ and $\prec^{[k]}$ to denote the
partial-order valued random elements with $\prec^\N(\omega) = <^\N$ and\vspace*{1pt}
$\prec^{[k]}(\omega) = <^\N_{[k]}$, and $\Pi$ and $\Pi_k$ to
denote the
posets induced on the random sets $\Xi$ and\vspace*{1pt} $\Xi_k$, respectively, by the
random order $\prec^\N$ and $\prec^{[k]}$, respectively; in other
words, $\Pi= (\Xi, \prec)$ and $\Pi_k = (\Xi_k,\prec)$, where
$\xi_i \prec\xi_j$ if and only if $i \prec^\N j$ or $i \prec^{[k]} j$.

Let $\cB$ denote the family of Borel subsets of $[0,1]$. For $k\in\N$,
sets $B_1, \ldots, B_k$ in $\cB$, and $<^{[k]}$ a partial order on $[k]$,
define $(B_1\cdots B_k,<^{[k]})$ to be the set of pairs
$(x_1\cdots x_k,<^{[k]})$ in $\cE^{[k]}$ with $x_i \in B_i$ for each $i$.
Now define
\[
E\bigl(B_1\cdots B_k,<^{[k]}\bigr) = \pi_k^{-1}\bigl(B_1\cdots B_k,<^{[k]}\bigr).
\]
This subset $E(B_1\cdots B_k,<^{[k]})$ of $\Omega$ is to be thought of as
the event that $\xi_i \in B_i$ for $i=1, \ldots, k$, and that $\prec
^\N$
has $<^{[k]}$ as its restriction to $[k]$. An event of
this form will be called a \textit{basic event}.

For each $k$, we now define $\cF_k$ to be the $\sigma$-field
generated by
the sets $E(B_1\cdots B_k,<^{[k]})$, and we note that $\cF_k \subseteq\cF_{k+1}$.
Clearly, the family of events $E(B_1\cdots B_k,<^{[k]})$ determines,
and is determined by, the first $k$ states of the causet process, so
the $\cF_k$ form the natural filtration for our process. We then take
$\cF= \sigma(\bigcup_{k=1}^{\infty} \cF_k)$. A causet process thus
gives rise to a probability measure $\mu$ on~$\cF$, that we will call a
\textit{causet measure}.

We remark that $\Omega$ can be identified formally with a subspace of the
compact space $[0,1]^\N\times2^\N$, with the product topology (and the
standard topology on $[0,1]$). Here, we take an enumeration
$s\dvtx \N\to\N\times\N$ of the set of pairs $(i,j)$ of positive integers
with $i<j$, and then encode a suborder $<^{\N}$ of $\N$ as a function
$q\dvtx \N \to\{0,1\}$ by setting $q(s^{-1}(i,j)) = 1$ if and only if
$i<^{\N} j$. The topological space $[0,1]^\N\times2^\N$ is metrisable,
for instance by the metric
\[
d((\mathbf{a,c}),(\mathbf{b,d})) = \sum_i 2^{-i} (|a_i - b_i| + |c_i - d_i|).
\]
The requirement that $<^\N$ be a partial order translates
to: $q(s^{-1}(i,j))+q(s^{-1}(j, k))-q(s^{-1}(i,k)) \le1$ for each $i<j<k$.
The subspace of $[0,1]^\N\times2^\N$ satisfying these constraints is
therefore closed, and hence compact.

In this representation of $\Omega$ as a product space, the $\sigma$-fields
$\cF_k$ contain all finite-dimensional sets. By separability, every open
set is a countable union of sets in $\bigcup_{k=1}^{\infty} \cF_k$, so the
product $\sigma$-field $\cF$ is the Borel $\sigma$-field on $\Omega
$ (see,
e.g., Remark~4.A3 in~\cite{Georgii} or the discussion of product
spaces in Chapter 3 in~\cite{EK}), so our causet measures will be Borel
measures. As $\Omega$ is a closed subset of a complete and separable
metric space, $\Omega$ itself is also complete and separable.

As we have already indicated, we shall treat the concepts of causet
measure and causet process almost interchangeably. Let us spell out why
we may do this.\vspace*{1pt}

The family of basic events $E(B_1\cdots B_k,<^{[k]})$ forms a
\textit{separating class}, that is, any two probability measures that agree
on all basic events are equal: see, for example, Proposition~4.6 in
Chapter~3 of~\cite{EK} or Example~1.2 in~\cite{billingsley}. Thus, to
specify a causet measure $\mu$ on $(\Omega,\cF)$, it suffices to specify
the probabilities $\mu(E(B_1\cdots B_k,<^{[k]}))$ in a consistent way.\vspace*{1pt}
Indeed, as mentioned earlier, the restriction $\mu_k =\mu\pi_k^{-1}$ of
$\mu$ to $\cF_k$ specifies the evolution of the process through the first
$k$ steps; the measures $\mu_k$ are the finite-dimensional distributions
of the process, and they determine the distribution of the process---see
Proposition 3.2 in~\cite{kallenberg} or Theorem~1.1 in~Chapter~4
of~\cite{EK}, or Example~1.2 in~\cite{billingsley}.

Conversely, suppose we are given the causet process as a transition
function $P(\cdot, \cdot)$, that is:
\begin{itemize}[(ii)]
\item[(i)] for each state $(x_1\cdots x_k,<^{[k]})$ in $\cE^{[k]}$,
$P((x_1\cdots x_k,<^{[k]}), \cdot)$ [the transition probability
from the state $(x_1\cdots x_k,<^{[k]})$] is a probability measure on
$\cE^{[k + 1]}$,
\item[(ii)] for every $k\in\N$, every $B_1, \ldots, B_{k+1} \in\cB$,
and every\vspace*{1pt} suborder $<^{[k + 1]}$ of $[k + 1]$,
$P(\cdot, (B_1\cdots B_{k+1},<^{[k + 1]}))$ is a
Borel-measurable function on $\cE^k$.
\end{itemize}
Then the probabilities $\mu(E(B_1\cdots B_k,<^{[k]}))$ can be derived
as integrals of products of evaluations of the transition function.
See Chapter 4 of Ethier and Kurtz~\cite{EK} for details.

One feature of our model that we have not built in to the space
$(\Omega,\cF)$ is the requirement that the labels on elements be
distinct: $\xi_i \not= \xi_j$ for each $i\not=j$. Indeed, it is
convenient to include elements with repeated labels in our sample space
$\Omega$, for instance, so that the space is compact. However, as we are
interested in processes that generate labeled causal sets, we do demand
that the transitions of a causet process are such that the probability of
choosing any element more than once is 0:
$\mu(\{\omega\dvtx \exists i\not= j, \xi_i(\omega)=\xi_j(\omega) \})=0$.

%s4 ###
\section{Order-invariant processes and measures} \label{sec:oip}

Causet processes, as defined above, are very general in nature. We are
principally interested in those satisfying the property of
\textit{order-invariance}, which we shall define shortly.

When we consider an element $\omega= (x_1x_2\cdots, <^\N)$ of
$\Omega$,
the real object of interest is the derived causal set $P = \Pi(\omega)$,
with ground-set $X = \Xi(\omega) = \{x_1, x_2, \ldots\}$. Suppose that
$x_{\lambda(1)}x_{\lambda(2)} \cdots$ is another natural extension
of $P$;
this means exactly that the permutation $\lambda$ of $\N$ is a natural
extension of~$<^\N$. There is just one suborder, which we shall denote
$\lambda[<^\N]$, of $\N$ with the property that\vspace*{1pt}
$(x_{\lambda(1)}x_{\lambda(2)} \cdots, \lambda[<^\N])$ induces
$P$. To
specify this order, note that we require $i(\lambda[<^\N])j$ if and only
if $x_{\lambda(i)} < x_{\lambda(j)}$ in $P$, which is equivalent to
$\lambda(i) <^\N\lambda(j)$.

Accordingly, for any $\omega= (x_1x_2\cdots, <^\N) \in\Omega$, and any
natural extension $\lambda$ of $<^\N$, we define the order $\lambda
[<^\N]$
by: $i(\lambda[<^\N])j$ if and only if $\lambda(i) <^\N\lambda(j)$.
We also define
$\lambda[\omega] = (x_{\lambda(1)}x_{\lambda(2)} \cdots, \lambda
[<^\N])$.
As we have seen, the elements $\omega$ and $\lambda[\omega]$ of
$\Omega$
give rise to the same poset $P$: in other words
$\Pi(\omega) = \Pi(\lambda[\omega])$.

With this definition, the permutation $\lambda$ of $\N$ acts on a subset
of $\Omega$. Our definition of order-invariance will demand, roughly,
that, whenever $\lambda$ is a~permutation of $\N$ that fixes all but
finitely many elements, and $\lambda$ acts bijectively on a suitable
subset $E$ of $\Omega$, then $\mu(\lambda[E]) = \mu(E)$.

We define similar notation for the case of finite posets with
ground-set~$[k]$. For a permutation $\lambda$ of~$[k]$, and $<^{[k]}$ a partial
order on $[k]$, let $\lambda[<^{[k]}]$\vspace*{1pt} be the partial order on $[k]$ given
by: $i(\lambda[<^{[k]}])j$ if and only if $\lambda(i) <^{[k]} \lambda(j)$.\vspace*{1pt}
The permutations $\lambda$ of $[k]$ such that $\lambda[<^{[k]}]$ is a
suborder of $[k]$ are exactly the linear extensions of $<^{[k]}$: those
where, whenever $\lambda(i) < \lambda(j)$, $i$ precedes $j$ in the
standard order on $[k]$.

A measure $\mu$ on $(\Omega,\cF)$ is \textit{order-invariant} if, for any
finite sequence $B_1,\ldots,\break B_k$ of sets in $\cB$, any suborder $<^{[k]}$
of $[k]$, and any linear extension $\lambda$ of $<^{[k]}$, we have
%
%e4 ###
\begin{equation} \label{order-invariance}
\mu\bigl(E\bigl(B_1\cdots B_k,<^{[k]}\bigr)\bigr) =
\mu\bigl(E\bigl(B_{\lambda(1)}\cdots B_{\lambda(k)},\lambda\bigl[<^{[k]}\bigr]\bigr)\bigr).
\end{equation}

To check that a process or measure is order-invariant, it is enough to
verify condition~(\ref{order-invariance}) for those $\lambda$ transposing
two adjacent incomparable elements. This is a consequence of the (easy)
fact that, given two linear extensions of a finite partial order, it is
possible to step from one to the other via a~sequence of transpositions of
adjacent incomparable elements: a proof of this is sketched in Example~\ref{ex1}.

In the case where the $B_i$ are singleton sets, (\ref{order-invariance})
says that the probability of a state $(x_1\cdots x_k, <^{[k]})$ depends
only on the set $X_k =\{ x_1,x_2, \ldots, x_k\}$ of elements, and the
partial order $P_k$ induced on $X_k$ by $<^{[k]}$, and not on the order
in which the elements of $X_k$ were generated. For instance, in
Example~\ref{ex1}, where the causet measure is prescribed by the probabilities of
single states, this can be taken as the definition of order-invariance,
which is exactly what we did in the \hyperref[sec:intro]{Introduction}. More typically, the
probability of any single state $(x_1\cdots x_k, <^{[k]})$ will be~0, so
the definition we used in Example~\ref{ex1} will not suffice.

A causet process whose distribution is given by an order-invariant
measure~$\mu$ on $(\Omega, \cF)$ is said to be an
\textit{order-invariant causet process}. As we saw at the end of the previous
section, we can talk about order-invariant measures and (distributions of)
order-invariant processes interchangeably.\vspace*{1pt}

As an example, suppose $k=3$ and $<^{[k]}$ has only one related pair,
$1 <^{[3]} 2$. Consider the linear extension $\lambda$ given by:
$\lambda(1)=3$, $\lambda(2)=1$ and $\lambda(3)=2$. Then
$2 (\lambda[<^{[3]}]) 3$ is the only related pair in $\lambda[<^{[3]}]$,
and this instance of the condition of order-invariance is that
\[
\mu\bigl(E\bigl(ABC,<^{[3]}\bigr)\bigr) = \mu\bigl(E\bigl(CAB,\lambda\bigl[<^{[3]}\bigr]\bigr)\bigr)
\]
for any $A,B,C \in\cB$. (Think of $A$, $B$ and $C$ as disjoint for
convenience.) On both sides the restriction is that the element in $A$
is below the element in~$B$ in the partial order $\Pi_3$, while the
element in $C$ is incomparable to both. The order-invariance condition
tells us that, conditioned on the event ``after three steps, we have an
element in $A$ below an element in $B$, and an element in $C$ incomparable
to both,'' each possible order of generation of the three elements is
equally likely. In this case, the possible orders of generation are just
the ones in which the element of $A$ precedes the element in $B$: besides
the two orders above, the only other possible order of generation is
$ACB$. The three orders correspond to the three linear extensions of the
poset $Q$ with three elements labeled $A$, $B$ and $C$, with $A$ below
$B$.

Another related condition we can impose on a causet process is that
transitions \textit{out of} a state depend only on the set of elements
generated and the partial order induced on them. Specifically, we say
that a causet process (or associated measure) is \textit{order-Markov} if we
always have
\[
\frac{\mu(E(B_1\cdots B_kB_{k+1},<^{[k + 1]}))}
{\mu(E(B_1\cdots B_k,<^{[k]}))} =
\frac{\mu(E(B_{\lambda(1)}\cdots B_{\lambda(k)}B_{k+1},
\lambda'[<^{[k + 1]}]))}
{\mu(E(B_{\lambda(1)}\cdots B_{\lambda(k)},\lambda[<^{[k]}]))},
\]
whenever either denominator is non-zero, where $\lambda'$ is the linear
extension of $<^{[k + 1]}$ derived from a linear extension $\lambda$ of
$<^{[k]}$ by fixing $k + 1$. We see immediately that, if a causet
process is order-invariant, then it is order-Markov, as the two numerators
and the two denominators above are both equal.

The converse is far from true: order-invariance is much stronger than the
order-Markov condition. One way to see this is to observe that, if we
impose only the order-Markov condition, then the transition laws out of
states with one element need bear no relation to the transition law out of
the initial ``empty'' state: if we demand order-invariance, then these are
connected via equation~(\ref{order-invariance}) in cases where $<^{[2]}$
is the two-element antichain.

However, if we know that a causet process is order-Markov, then to prove
order-invariance it is enough to check condition~(\ref{order-invariance})
for the permutation $\lambda$ exchanging the \textit{last} two elements,
whenever these are incomparable. To see this, let $\lambda^{(i)}$ denote
the permutation of any $[k]$, with $k>i$, exchanging $i$ and $i+1$ and
leaving all other elements fixed. Suppose that the causet measure~$\mu$
satisfies
\[
\mu\bigl(E\bigl(B_1\cdots B_{j-2}B_jB_{j-1}, \lambda^{(j-1)}\bigl[<^{[j]}\bigr]\bigr)\bigr) =
\mu\bigl(E\bigl(B_1\cdots B_{j-2}B_{j-1}B_j, <^{[j]}\bigr)\bigr)
\]
for every sequence $B_1,\ldots,B_j$ of Borel sets, and every suborder
$<^{[j]}$ of $[j]$ in which $j-1$ and $j$ are incomparable. Now if $\mu$
is order-Markov, we can use this condition inductively to deduce that
\begin{eqnarray*}
&&\mu\bigl(E\bigl(B_1\cdots B_{j-2}B_jB_{j-1}B_{j+1}\cdots B_k,
\lambda^{(j-1)}\bigl[<^{[k]}\bigr]\bigr)\bigr) \\
&&\qquad =
\mu\bigl(E\bigl(B_1\cdots B_{j-2}B_{j-1}B_jB_{j+1}\cdots B_k, <^{[k]}\bigr)\bigr)
\end{eqnarray*}
for every $k \ge j$, every sequence $B_1,\ldots,B_k$ of Borel sets, and
every suborder $<^{[k]}$ of $[k]$ in which $j-1$ and $j$ are incomparable.
This is exactly condition~(\ref{order-invariance}) for $\lambda^{(j-1)}$.
As we saw earlier, we can now deduce that $\mu$ is order-invariant.

Given two probability measures $\mu_1$ and $\mu_2$ on $(\Omega,\cF)$,
a \textit{convex combination} of $\mu_1$ and $\mu_2$ is a probability
measure of the form $r \mu_1 + (1-r) \mu_2$, for $r \in(0,1)$. It is
immediate from the definition that, if $\mu_1$ and $\mu_2$ are
order-invariant, then so is any convex combination. Thus, the family
of order-invariant measures is a convex subset of the set of all causet
measures.

More generally, given a probability space $(W,\cG,\rho)$ whose elements
are causet measures $\mu_\omega$, the \textit{mixture} defined by this space
is the probability measure $\mu$ defined by
\[
\mu(\cdot) = \int_W \mu_\omega(\cdot)  \, d\rho(\mu_\omega).
\]
It is again immediate from the definition that, if all the $\mu_\omega
$ are
order-invariant, then so is the mixture $\mu$.

We now give an alternative characterization of order-invariance. For this,
we need to introduce some notation that will feature prominently in the
subsequent sections as well.

For $<^\N$ a suborder of $\N$, $k \in\N$, and $\lambda$ a linear\vspace*{-1pt}
extension of $<^\N_{[k]}$, we define~$\lambda^+$ to be the natural
extension of $<^\N$ defined by
\[
\lambda^+(i) =
\cases{
\lambda(i), &\quad  $i \le k$, \cr
i, &\quad $ i > k$.
}
\]
So $\lambda^+[<^\N]$ is the partial order on $\N$ obtained from
$<^\N$
obtained by permuting the first $k$ labels according to $\lambda$.

For a fixed $\omega= (x_1x_2\cdots, <^\N) \in\Omega$, $k \in\N$, and
$E \in\cF$, we define $\nu^k(E)(\omega)$ as the proportion of linear
extensions $\lambda$ of $<^\N_{[k]}$ such that $\lambda^+[\omega]$
is in
$E$.

For any $\omega\in\Omega$ and $k$, the function $\nu^k(\cdot
)(\omega)$
gives a probability measure on $\cF$, namely the uniform measure on
elements $\lambda^+[\omega]$ of $\Omega$, where $\lambda$ runs over linear
extensions of $<^\N_{[k]}$. This measure can naturally be identified with
the uniform measure on linear extensions of $\Pi_k(\omega)$.

Let us look more closely at $\nu^k(E(B_1 \cdots B_n, <^{[n]}))(\omega)$,
where $\omega= (x_1x_2 \cdots,\break <^\N)$, the $B_i$ are Borel sets in
$[0,1]$, $<^{[n]}$ is a suborder of $[n]$, and $k\le n$. In order for
this quantity to be non-zero, it is necessary for $<^\N_{[k]}$ to be
isomorphic to $<^{[n]}_{[k]}$.\vspace*{-2pt}

Suppose the poset $<^{[n]}_{[k]}$ has $\ell$ linear extensions,\vspace*{1pt}
$\lambda_1, \ldots, \lambda_{\ell}$. Each $\lambda_i$ induces a linear
extension $\lambda'_i$ on $<^{[n]}$, obtained by fixing the elements
$k+1, \ldots, n$.

If now $\nu^k(E(B_1 \cdots B_n, <^{[n]})(\omega)$ is non-zero, then
$<^\N_{[k]}$ also has $\ell$ linear\vspace*{-1pt} extensions, and
$\nu^k(E(B_1 \cdots B_n, <^{[n]}))(\omega)$ is equal to $\frac
{1}{\ell}$
times the number of them that, applied to $\omega$, yield an element
in $E(B_1 \cdots B_n, <^{[n]})$. If, for some\vspace*{1pt} linear extension
$\rho$ of $<^\N_{[k]}$, $\rho^+[\omega]$ is in the set
$E(B_1 \cdots B_n, <^{[n]})$, then we can reverse the process: for one of
the linear extensions $\lambda_i$, $\lambda^+_i[\rho^+[\omega]] =
\omega$.
In other words, $\rho$ has to be the inverse of one of the $\lambda_i$,
and the set of $\omega$ for which $\rho^+[\omega]$ is in
$E(B_1 \cdots B_n, <^{[n]})$ is the set
$E(B_{\lambda'_i(1)} \cdots B_{\lambda'_i(n)}, \lambda'_i[<^{[n]}])$.

It now follows that
%
%e5 ###
\begin{equation} \label{eq:sum}
\nu^k\bigl(E\bigl(B_1\cdots B_n, <^{[n]}\bigr)\bigr)(\omega) = \frac{1}{\ell}
\sum_{i=1}^{\ell} \bone\bigl(E\bigl(B_{\lambda'_i(1)} \cdots B_{\lambda'_i(n)},
\lambda'_i\bigl[<^{[n]}\bigr]\bigr)\bigr)(\omega).
\end{equation}

\begin{lemma} \label{lem:lambda-k-invariant}
For any $k \in\N$, any Borel sets $B_1, \ldots, B_k$, any suborder $<^{[k]}$
of $[k]$, any linear extension $\lambda$ of $<^{[k]}$, and any
$\omega\in\Omega$, we have
\[
\nu^k\bigl(E\bigl(B_1 \cdots B_k, <^{[k]}\bigr)\bigr)(\omega)=
\nu^k\bigl(E\bigl(B_{\lambda(1)}\cdots B_{\lambda(k)}, \lambda
\bigl[<^{[k]}\bigr]\bigr)\bigr)(\omega).
\]
\end{lemma}

\begin{pf}
$\!\!$For $\omega= (x_1x_2\cdots, \prec^\N)\in\Omega$, the quantity\vspace*{1pt}
\mbox{$\nu^k(E(B_1 \cdots B_k, <^{[k]}))(\omega)$} is the proportion of linear
extensions $\rho$ of $\prec^\N_{[k]}$ such that
$\rho^+[\omega] \in E(B_1\cdots B_k,\break <^{[k]})$. We see that
$\rho^+[\omega] \in E(B_1\cdots B_k, <^{[k]})$ if and only if
$(\lambda\rho)^+[\omega] \in\break
E(B_{\lambda(1)}\times \cdots B_{\lambda(k)}, \lambda[<^{[k]}])$.
Therefore, as required, we have
\[
\nu^k\bigl(E\bigl(B_1 \cdots B_k, <^{[k]}\bigr)\bigr)(\omega)=
\nu^k\bigl(E\bigl(B_{\lambda(1)}\cdots B_{\lambda(k)}, \lambda
\bigl[<^{[k]}\bigr]\bigr)\bigr)(\omega).
\]
\upqed
\end{pf}

We are now in a position to establish our alternative characterization of
order-invariance.

\begin{theorem} \label{thm:DLR}
Let $\mu$ be a causet measure. Then $\mu$ is order-invariant if and only
if
\[
\mu(E) = \E_\mu\nu^k(E)
\]
for every $E\in\cF$ and every $k\in\N$.
\end{theorem}

This is an analogue of the DLR equations from statistical physics (see,
e.g., Bovier~\cite{Bovier}) or Section 1.2 in~\cite{Georgii},
about conditional probabilities. It corresponds to specifying a
\textit{boundary condition} outside a \textit{finite volume}---here this means
that we condition on all the information about $\omega$ except the order
in which the first $k$ elements are generated, and then realizing the
\textit{conditional Gibbs measure}, which in our setting is
$\nu^k(\cdot)(\omega)$.

\begin{pf*}{Proof of Theorem \ref{thm:DLR}}
Suppose first that $\mu$ is a causet measure satisfying the condition.
Consider any finite sequence $B_1, \ldots, B_k$ of sets in~$\cB$, any
suborder $<^{[k]}$ of $[k]$, and any linear extension $\lambda$ of
$<^{[k]}$.

By Lemma~\ref{lem:lambda-k-invariant}, we have that
\[
\nu^k\bigl(E\bigl(B_1 \cdots B_k, <^{[k]}\bigr)\bigr)(\omega)=
\nu^k\bigl(E\bigl(B_{\lambda(1)}\cdots B_{\lambda(k)}, \lambda
\bigl[<^{[k]}\bigr]\bigr)\bigr)(\omega)
\]
for each $\omega\in\Omega$. Taking expectations, we have
\[
\E_\mu\nu^k\bigl(E\bigl(B_1 \cdots B_k, <^{[k]}\bigr)\bigr) =
\E_\mu\nu^k\bigl(E\bigl(B_{\lambda(1)}\cdots B_{\lambda(k)}, \lambda\bigl[<^{[k]}\bigr]\bigr)\bigr),
\]
and therefore the given condition implies that
\[
\mu\bigl(E\bigl(B_1 \cdots B_k, <^{[k]}\bigr)\bigr) =
\mu\bigl(E\bigl(B_{\lambda(1)}\cdots B_{\lambda(k)}, \lambda\bigl[<^{[k]}\bigr]\bigr)\bigr)
\]
for each basic event $E(B_1 \cdots B_k, <^{[k]})$, as required for
order-invariance.

Conversely, suppose that $\mu$ is order-invariant, and fix $k \in\N$.
Now consider a basic event $E(B_1\cdots B_n, <^{[n]})$, for $n \ge k$.
Taking expectations in (\ref{eq:sum}), we obtain that
\[
\E_\mu\nu^k\bigl(E\bigl(B_1\cdots B_n, <^{[n]}\bigr)\bigr) = \frac{1}{\ell}
\sum_{i=1}^{\ell} \mu\bigl(E\bigl(B_{\lambda'_i(1)} \cdots B_{\lambda'_i(n)},
\lambda'_i\bigl[<^{[n]}\bigr]\bigr)\bigr),
\]
where, as before, $\lambda'_1, \ldots, \lambda'_\ell$ are the linear\vspace*{-1pt}
extensions of $<^{[n]}_{[k]}$ that fix $k+1, \ldots, n$.
Now, by order-invariance, the sum above is equal to
$\mu(E(B_1 \cdots B_n,\break <^{[n]}))$.

For each fixed $k$, we now have that $\mu(E) = \E_\mu\nu^k(E)$ for all
the basic events $E$; as both $\mu(\cdot)$ and $\E_\mu\nu^k(\cdot
)$ are
measures, and the basic events form a separating class, we have that the
condition holds for all events $E\in\cF$.
\end{pf*}

%s5 ###
\section{Examples} \label{sec:tsc}

In this section, we briefly revisit the three examples we introduced in
Section~\ref{sec:intro}, and give one more.

\subsection*{Causet processes on fixed causal sets}
Suppose we are given a fixed causal set $P=(Z,<)$, with $Z \subset[0,1]$.
Recall that $L'(P)$ is the set of all natural extensions of posets $P_Y$,
where $Y$ is an infinite down-set in $P$. In the case where the set
$I(x)$ of elements incomparable to $x$ is finite for all $x$, $L'(P)$ is
equal to the set $L(P)$ of natural extensions of $P$.

A \textit{causet process on $P$} is a process generating a random element
$\lambda$ of $L'(P)$: we think of generating distinct elements
$\lambda(1), \lambda(2), \ldots$ of $Z$ in turn. At each stage, an element
$z \in Z$ is available for selection only if all the elements in $D(z)$
have already been selected [equivalently, at stage $k$, the element $z$
is available for selection if $z$ is minimal in
$P\sm\{\lambda(1), \ldots, \lambda(k-1)\}$].

We can view a causet process on $P$ as a special case of a causet process:
the states that can occur are pairs $(x_1\cdots x_k, <^{[k]})$, where
$x_1\cdots x_k$ is an ordered stem of $P$, and $<^{[k]}$ is the poset
induced from the order $<$ on $Z$: $i<^{[k]}j$ if and only if $x_i < x_j$.
For a transition out of this state, at stage $k + 1$, a~random (not
necessarily uniform) minimal element $\xi_{k+1}$ of
$P\sm\{x_1,\ldots,x_k\}$ is selected, and its down-set is chosen to be the
same as it is in $P$. Example~\ref{ex1} illustrates this.

A causet process on $P$ is order-Markov if the law describing how we
choose a minimal element from $P\sm\{x_1,\ldots,x_k\}$ depends only on the
set $\{x_1,\ldots,x_k\}$. Again, the condition of order-invariance is much
more demanding than the order-Markov condition.

One example is the process considered by Luczak and Winkler~\cite{LW},
which grows, step by step, uniformly random $n$-element subtrees,
containing the root as the unique minimal element, of the complete
$d$-ary tree $T^d$. This process is a causal set process on $T^d$, and is
order-Markov, but calculations on small examples reveal that it is not
order-invariant.

When considering causet processes on a fixed poset $P$, the order on any
set of elements is determined by $P$, so it is natural to drop the order
from the notation, and denote a state simply as $x_1\cdots x_k$, and an
element $\omega$ as $x_1x_2\cdots$.

We study order-invariance on fixed causal sets in more detail in the
companion paper~\cite{BL2}.

We saw one example of a causet process on a fixed causal set in Example~\ref{ex1}.
As a further illustration, we give another example.

\begin{ex}\label{ex4}
Let $P=(Z,<)$ be the disjoint union of two infinite chains
$B\dvtx b_1<b_2<\cdots$ and $C\dvtx c_1<c_2<\cdots$, with every element of $B$
incomparable with every element of $C$. Fix a real parameter
$q \in[0,1]$, and define a~causet process on $P$ as follows. For any
stem $A$ of $P$, there are exactly two minimal elements of $P\sm A$, one
in $B$ and one in $C$: from any state with $X_k=A$, we define the
transition probabilities out of that state by choosing the element in $B$
with probability $q$. Denote the associated probability measure~$\mu_q$:
specifically, if $a_1\cdots a_k$ is any ordered stem of $P$ with
$A=\{a_1,\ldots,a_k\} = \{b_1,\ldots,b_m,c_1,\ldots,c_n\}$, then
$\mu_q(E(a_1\cdots a_k)) = q^m(1-q)^n$. (As mentioned above, we have\vspace*{1.5pt}
dropped the order $<^{[k]}$, which can be derived from $P$, from the
notation.)

It follows that $\mu_q$ is order-invariant for any $q$, since the
expression for $\mu_q(E(a_1\cdots a_k))$ does not depend on the order of
the $a_i$.

The cases $q=0$ and $q=1$ are special. If $q=0$, then elements from $C$
are never chosen, and $\Xi=B$ a.s.; if $q=1$, then $\Xi=C$ a.s.
If $q\in(0,1)$, then $\Xi=B\cup C =Z$ a.s.

More generally, given any probability measure $\rho$ on $[0,1]$,
define a
probability measure $\mu_\rho$ by first choosing a random parameter
$\chi$
according to $\rho$, then sampling according to $\mu_\chi$. Then, for
$a_1\cdots a_k$ is any ordered stem of $P$ with
$A=\{a_1,\ldots,a_k\} = \{b_1,\ldots,b_m,c_1,\ldots,c_n\}$, we have
$\mu_\chi(E(a_1\cdots a_k)) = \E_\rho(\chi^m (1-\chi)^n)$. Again, this
expression is independent of the order of the $a_i$, so $\mu_\rho$ is
order-invariant. (Alternatively, $\mu_\rho$ is a mixture of the
order-invariant measures $\mu_q$, so is also order-invariant.)

This last description includes several apparently different processes.
For instance, consider the following process: having chosen the bottom $n$
elements, $m$ from $B$ and $k=n-m$ from $C$, choose the next element to be
from $B$ with probability $(m+1)/(n+2)$. It is easy to check directly
that this defines an order-invariant process on $P$. The theory of
\textit{P\'olya's Urn} (see, e.g., Exercise~E10.1 in
Williams~\cite{williams}) tells us that the proportion of elements taken
from $B$ in the first $n$ steps converges a.s.\ to some limit $\chi$ as
$n\to\infty$, and that this limit $\chi$ has the uniform
distribution on
$(0,1)$. Indeed, this process has the same finite-dimensional
distributions as the one defined by choosing $\chi$ from the uniform
distribution in advance, then choosing the natural extension according
to~$\mu_\chi$.

This example is covered in more detail in~\cite{BL2}, and from a slightly
different perspective in Kerov~\cite{Kerov}.
\end{ex}

\subsection*{Causet processes with independent labels}
Another special class of causet processes consists of those where, at
every transition, the new random ``label'' $\xi_{k+1}$ in $[0,1]$ is chosen
independently of the random down-set $\Delta_{k+1}$, and of all other
labels, and where the distribution of $\Delta_{k+1}$ itself depends only
on $<^{[k]}$.

In this case, the labels from $[0,1]$ play no essential role, and it is
more natural to think of the elements as unlabeled, and to view a process
as a Markov chain on the set of finite unlabeled causal sets. The csg
models of Rideout and Sorkin~\cite{RS}, which include the random graph
orders in Example~\ref{ex2}, are of this type.

Let us be specific about how to realize the random graph order with
parameter $p$ as an order-invariant causet process. The case $p=0$,
where the random graph order is a.s.\ an antichain, as in Example~\ref{ex2}, is
included in this description.

From a state $(x_1\cdots x_k, <^{[k]})$, we make a transition to a state
$(x_1 \cdots x_k x_{k+1},\break <^{[k+1]})$, where $x_{k+1}$ is chosen uniformly
at random from $[0,1]$, independent of any other choices. We choose
a random subset $\Sigma$ of $[k]$, with each element of $[k]$
appearing in
$\Sigma$ independently with probability~$p$. Now we define the down-set
$D(k+1)$ to be the set of elements $i \in[k]$ with $i\le^{[k]}j$ for some
$j \in\Sigma$.

We showed in Section~\ref{sec:intro} that the probability that the random\vspace*{-1pt}
partial order $\prec_{[k]}$ is equal to a particular suborder
$<^{[k]}$ of
$[k]$ is given by $p^{c(<^{[k]})}(1-p)^{b(<^{[k]})}$, where $c(<^{[k]})$
is the number of covering pairs of $([k],<^{[k]})$, and $b(<^{[k]})$ is
the number of incomparable pairs. Thus, for $B_1, \ldots, B_k$ Borel sets
in $[0,1]$, and $<^{[k]}$ a suborder of $[k]$, we have
\[
\mu\bigl(E\bigl(B_1\cdots B_k,<^{[k]}\bigr)\bigr) =
|B_1|\cdots|B_k| p^{c(<^{[k]})}(1-p)^{b(<^{[k]})},
\]
where $|\cdot|$ denotes Lebesgue measure. The product
$|B_1|\cdots|B_k|$ is independent of the order of the $B_i$, and the
quantity $p^{c(<^{[k]})}(1-p)^{b(<^{[k]})}$ is invariant under
isomorphisms of the poset, so the measure $\mu$ is order-invariant.

The two special cases discussed above are, in a way, two extremes. When
we have a causal set process on a fixed causal set, the label of each
element determines its down-set when it is introduced: in the case of
causet processes with independent labels, the label and the down-set of an
element are independent.\vspace*{25pt}

%s6 ###
\section{Invariant measures} \label{sec:invariant}

In this section, we develop some weaker notions of invariance, and show
how these relate to order-invariance. One goal is to show that the family
of order-invariant measures is a closed subset of the family of all
probability measures on $(\Omega, \cF)$ with respect to the topology of
weak convergence.

For $i \in\N$, let $\lambda^{(i)}$ be the permutation of $\N$ exchanging
$i$ and $i + 1$:
\[
\lambda^{(i)}(j) =
\cases{
i+1, &\quad  if $j=i$, \cr
i, &\quad  if $ j=i+1$, \cr
j, &\quad  otherwise.
}
\]\vfill\eject

For $\omega= (x_1 x_2 \cdots, <^{\N}) \in\Omega$, a special case
of a
definition from Section~\ref{sec:oip} is that
\[
\lambda^{(i)} [\omega] = \bigl(x_{\lambda^{(i)}(1)} x_{\lambda
^{(i)}(2)}\cdots,
\lambda^{(i)}[<^\N]\bigr)
\]
whenever $\lambda^{(i)}$ is a natural extension of $<^\N$, that is, whenever
$i$ and $i + 1$ are incomparable in $<^\N$. We now extend
$\lambda^{(i)}$ to a function $\lambda^{(i)}\dvtx \Omega\to\Omega$ by
setting $\lambda^{(i)} [\omega]= \omega$ if the permutation $\lambda
$ is
not a natural extension of $<^\N$, that is, if $i <^\N i + 1$. Note that
each $\lambda^{(i)}$ is a permutation, indeed an involution, on
$\Omega$.

Observe that each $\lambda^{(i)}$ is continuous with respect to the
product topology on $\Omega$, and so is certainly measurable, as $\cF
$ is
the Borel $\sigma$-field with respect to this topology.

For $E\in\cF$, and $i \in\N$, we naturally define
$\lambda^{(i)}(E) = \{ \lambda^{(i)}[\omega] \dvtx \omega\in E\}$. Given
also a causet measure $\mu$, we set
$(\mu\circ\lambda^{(i)})(E) = \mu(\lambda^{(i)}(E))$. It is then
straightforward to check that $\mu\circ\lambda^{(i)}$ is a causet measure
for each $\mu$ and $i$.

\begin{lemma} \label{lem:invt}
For each $i$, a causet measure $\mu$ satisfies
$\mu= \mu\circ\lambda^{(i)}$ if and only if
\[
\mu\bigl(E\bigl(B_1 \cdots B_iB_{i+1} \cdots B_k, <^{[k]}\bigr)\bigr) =
\mu\bigl(E\bigl(B_1 \cdots B_{i+1}B_i \cdots B_k, \lambda^{(i)}\bigl[<^{[k]}\bigr]\bigr)\bigr)
\]
for all $k >i$, all Borel sets $B_1, \ldots, B_k$, and all suborders
$<^{[k]}$ of $[k]$ such that~$i$ and $i+1$ are incomparable.
\end{lemma}

\begin{pf}
The given condition amounts to saying that the two measures~$\mu$ and
$\mu\circ\lambda^{(i)}$ agree on all the basic events
$E = E(B_1 \cdots B_k, <^{[k]})$ with $k>i$\break and $i$ and $i + 1$
incomparable in $<^{[k]}$. This also holds trivially for those $<^{[k]}$
with $i <^{[k]} i+1$, so the condition is equivalent to the statement
that the two measures agree on the separating class of all basic events
$E = E(B_1 \cdots B_k, <^{[k]})$ with $k>i$.
\end{pf}

For $k \in\N$, let
$\Lambda^k = \{ \lambda^{(1)}, \ldots, \lambda^{(k-1)} \}$.
Also, set $\Lambda= \bigcup_k \Lambda^k = \{ \lambda^{(i)} \dvtx i\in
\N\}$.

We say that a measure $\mu$ is {\em$\Lambda^k$-invariant} if
$\mu= \mu\circ\lambda^{(i)}$ for each $\lambda^{(i)} \in\Lambda
^k$; we
say that $\mu$ is {\em$\Lambda$-invariant} if
$\mu\circ\lambda^{(i)} = \mu$ for every $i$.

The following result is now immediate from Lemma~\ref{lem:invt}.

\begin{theorem}
$\!\!\!\!\!$A measure is $\Lambda$-invariant if and only if it is order-invariant.
\end{theorem}

\begin{pf}
From Lemma~\ref{lem:invt}, we see that a measure is $\Lambda$-invariant
if and only if $\mu$ satisfies (\ref{order-invariance}) whenever
$\lambda$
is one of the $\lambda^{(i)}$, that is, whenever $\lambda$ exchanges two
adjacent incomparable elements. As we remarked immediately after the
definition of order-invariance, this special case implies that~(\ref{order-invariance})
holds for all $\lambda$, that is, that $\mu$ is order-invariant.
\end{pf}

For a fixed $\omega= (x_1x_2\cdots, <^\N) \in\Omega$, $k \in\N$, and
$E \in\cF$, recall that $\nu^k(E)(\omega)$ is the proportion of linear
extensions $\lambda$ of $<^\N_{[k]}$ such that $\lambda^+[\omega]$ is
in~$E$.\vadjust{\eject}

We will now show that the measures $\nu^k(\cdot)(\omega)$ are
$\Lambda^k$-invariant, a result closely related to
Lemma~\ref{lem:lambda-k-invariant}. (To be precise, the special case of
that lemma with $\lambda=\lambda^{(i)}$ and $i<k$ is also a special case
of the following result, and from that special case it is easy to deduce
Lemma~\ref{lem:lambda-k-invariant}.)

\begin{theorem} \label{thm:lambda-k-invariant}
For each $k \in\N$ and $\omega\in\Omega$, the measure
$\nu^k(\cdot)(\omega)$ is $\Lambda^k$-inva\-riant.
\end{theorem}

\begin{pf}
Fix $k \in\N$ and $\omega= (x_1x_2\cdots, \prec^\N)\in\Omega$.
We have to show that $\nu^k(E)(\omega) = \nu^k(\lambda
^{(i)}(E))(\omega)$
for every $E\in\cF$ and every $i < k$.

For $i<k$, we can consider $\lambda^{(i)}$ as acting on the set of linear\vspace*{-1pt}
extensions $\rho$ of $\prec^\N_{[k]}$ as follows. If $i$ and $i+1$ are
incomparable in $\rho[\prec^\N_{[k]}]$---that is, if $\rho(i)$ and\vspace*{-1pt}
$\rho(i+1)$ are incomparable in $\prec^\N$---then
$\lambda^{(i)}[\rho] = \lambda^{(i)}\circ\rho$; if $i$ and $i +
1$ are
comparable in $\rho[\prec^\N_{[k]}]$, then $\lambda^{(i)}[\rho] =
\rho$.
Thus, $\lambda^{(i)}$ acts as an involution\vspace*{-1pt} on the set of linear extensions
of $\prec^\N_{[k]}$.

We claim that
\[
\bigl(\lambda^{(i)}[\rho]\bigr)^+[\omega] = \lambda^{(i)}[\rho^+[\omega]].
\]
If $i$ and $i+1$ are comparable in $\rho[\prec^\N_{[k]}]$---that
is, if
$\rho(i) \prec^\N_{[k]} \rho(i+1)$---then both are equal to
$\rho^+[\omega]$. If $i$ and $i+1$ are incomparable in
$\rho[\prec^\N_{[k]}]$,\vspace*{-1pt} then both are obtained from
$\rho^+[\omega] =(x_{\rho^+(1)}x_{\rho^+(2)}\cdots, \rho[\prec
^\N_{[k]}])$
by exchanging the terms $x_{\rho^+(i)}$ and $x_{\rho^+(i+1)}$ and
changing the order to
$\lambda^{(i)}[\rho[\prec^\N_{[k]}]] =
(\lambda^{(i)}\circ\rho)[\prec^\N_{[k]}]$: to see that these
orders are
equal, note that $j < \ell$ in each order if and only if
$\lambda^{(i)}\rho(j) \prec^\N_{[k]} \lambda^{(i)}\rho(\ell)$.

For a linear extension $\rho$ of $\prec^\N_{[k]}$, we see that
$\rho^+[\omega] = (x_{\rho(1)}x_{\rho(2)} \cdots\break
x_{\rho(k)}x_{k+1}\times
\cdots, \rho^+[\prec^\N])$ is in $E$ if and only if
$(\lambda^{(i)}[\rho])^+[\omega] = \lambda^{(i)}[\rho^+[\omega]]$
is in
$\lambda^{(i)}(E)$.

Therefore, the proportion of linear extensions $\rho$ of $\prec^\N_{[k]}$\vspace*{-1pt}
such that $\rho^+[\omega] \in E$ is the same as the proportion of linear
extensions $\lambda^{(i)}[\rho]$ such that\break
$(\lambda^{(i)}[\rho])^+[\omega] \!\!\in\lambda^{(i)}(E)$, and this is the
desired result.
\end{pf}

For $k \in\N$, define $\cG_k$ to be the family of sets $H$ in $\cF$ such
that $\omega\in H$ implies $\lambda^{(i)}[\omega] \in H$ for each $i<k$.
We can alternatively write
\begin{eqnarray*}
\cG_k &=& \{ H \in\cF\dvtx \lambda(H)= H \mbox{ for all } \lambda\in
\Lambda^k\}\\
& =& \{ H \in\cF\dvtx \lambda^{-1}(H)= H \mbox{ for all }
\lambda
\in\Lambda^k\};
\end{eqnarray*}
that is, $\cG_k$ is the family of {\em$\Lambda^k$-invariant} sets.

It is easy to check that each $\cG_k$ is a $\sigma$-field, and also that
$\cG_{k+1} \subseteq\cG_k$ for all $k$. These $\sigma$-fields $\cG_k$
can be seen to correspond to the \textit{external $\sigma$-fields} in
Section~1.2 of~\cite{Georgii}.

Let
\[
\cG= \bigcap_{k=1}^\infty\cG_k = \bigl\{ H \in\cF\dvtx \lambda^{(i)}(H) = H
\mbox{ for all } i\bigr\}
\]
be the \textit{tail $\sigma$-field}, and call a set in $\cG$ a
\textit{tail event}.

An equivalent definition of $\cG_k$ is that it is the collection of sets
$H \in\cF$ such that $\omega= (x_1x_2\cdots, <^\N) \in H$ implies
$\lambda^+[\omega] \in H$ for every linear extension $\lambda$ of
$<^\N_{[k]}$. To see this, note again that, if $\lambda$ is a linear
extension of $<^\N_{[k]}$, then $\lambda$ can be generated from a sequence
of transpositions of adjacent incomparable elements, so there is a
sequence $\omega= \omega_0,\omega_1,\ldots,\omega_m = \lambda
^+[\omega]$
of elements of $\Omega$ such that, for each $j$,
$\omega_j = \lambda^{(i)}[\omega_{j-1}]$ for some $i < k$. If
$H \in\cG_k$, and $\omega\in H$, then each element of the sequence is
also in $H$.

Let us now consider the effect of conditioning on the $\sigma$-field
$\cG_k$.

As usual, for $E \in\cF$ and $\cH$ a sub-$\sigma$-field of $\cF$, we
shall write $\mu(E \mid\cH)$ for the conditional expectation
$\E_\mu(\bone_E \mid\cH)$.

\begin{theorem} \label{thm:nu-k}
$\!\!$Let $\mu$ be an order-invariant measure. For any event \mbox{$E \in\cF$}, and
any $k \in\N$, we have
\[
\mu(E \mid\cG_k) = \nu^k(E)
\]
almost surely.
\end{theorem}

\begin{pf}
Fix $E \in\cF$ and $k\in\N$.

We start by showing that $\nu^k(E)(\cdot)$ is $\cG_k$-measurable. Fix
now $\omega= (x_1x_2\cdots,\break <^\N) \in\Omega$ and $i<k$. We claim that
$\nu^k(E)(\omega) = \nu^k(E)(\lambda^{(i)}[\omega])$: this will
imply that
$\{ \omega\dvtx \nu^k(E)(\omega) \le x \}$ is in $\cG_k$, for all $x$, as
required.

The result is immediate unless $i$ and $i + 1$ are incomparable in
$<^\N$, so we suppose that they are incomparable. Then, for each linear
extension $\rho$ of $<^\N_{[k]}$, $\rho^+[\omega]$ is in $E$ if and only
if the corresponding permutation $\rho\lambda^{(i)}$ of $<^\N_{[k]}$ is
such that $(\rho\lambda^{(i)})[\lambda^{(i)}[\omega]] \in E$. Therefore,
we indeed have $\nu^k (E)(\omega) = \nu^k (E)(\lambda^{(i)} [\omega])$.

To show that the conditional expectation $\E_\mu(\bone_E \mid\cG
_k)$ is
equal to $\nu^k(E)$, we need to show
that
%
%e6 ###
\begin{equation} \label{eq:nu-k}
\int_H \bone_E \, d\mu= \int_H \nu^k(E) \, d\mu
\end{equation}
for every $H \in\cG_k$.
The left-hand side in (\ref{eq:nu-k}) is just $\mu(H \cap E)$.

The right-hand side is $\E_\mu(\bone_H \nu^k(E))$. We claim that
$(\bone_H \nu^k(E))(\omega) = \nu^k(H\cap E)(\omega)$ for all
$\omega\in\Omega$. Indeed, both sides are equal to $\nu^k(E)$ if
$\omega\in H$, and equal to zero if not, as $H$ is an invariant set for
$\Lambda_k$. Hence,  Theorem~\ref{thm:DLR} yields
\[
\E_\mu(\bone_H \nu^k(E)) = \E_\mu\bigl(\nu^k(H\cap E)\bigr) = \mu(H\cap E)
\]
as required.\vadjust{\eject}
\end{pf}

In the terminology of~\cite{Georgii}, Theorem~\ref{thm:nu-k} establishes
that the functions $\nu^k (\cdot) (\cdot)$ are a family of measure
kernels, analogous to Gibbsian specifications in statistical mechanics.

Our next aim is to show that the family of order-invariant measures, and
the family of $\Lambda^k$-invariant measures for each fixed $k$, are
closed subsets of the set of all causet measures, in the topology of weak
convergence.

Recall that a sequence $(\mu_n)$ of probability measures on a space
$(\Omega,\cF)$, where $\cF$ is the Borel $\sigma$-field for some topology
on $\Omega$, is said to \textit{converge weakly} to a probability measure
$\mu$ if $\E_{\mu_n} f \to\E_\mu f$, as $n \to\infty$, for every bounded
continuous real-valued function $f$ on $\Omega$: we write
$\mu_n \Rightarrow\mu$. There are a number of equivalent conditions---see Theorem~2.1 of Billingsley~\cite{billingsley}, for example. The one
that we shall make use of shortly is that $\mu_n \Rightarrow\mu$ if and
only if $\limsup\mu_n(F) \le\mu(F)$ for all closed sets $F$.
Another fact that we shall use later is that $\mu_n \Rightarrow\mu$ if
and only if $\mu_n(E) \to\mu(E)$ for all sets $E \in\cF$ such that
$\mu(\partial E) = 0$, where $\partial E$ denotes the boundary of~$E$.

We have already seen that our $\sigma$-field $\cF$ is the Borel
$\sigma$-field for the product topology on~$\Omega$.

Let $\cM$ be the set of probability measures on $(\Omega, \cF)$, let
$\cP$ be the set of those measures in $\cM$ that are order-invariant, and,
for $k \in\N$, let $\cP_k$ be the set of measures in $\cM$ that are
$\Lambda^k$-invariant.

\begin{theorem} \label{thm:closed}
For each $i$, $\{ \mu\dvtx \mu= \mu\circ\lambda^{(i)} \}$ is closed in the
topology of weak convergence. As a consequence, each of the $\cP_k$, and
$\cP$, are closed in the topology of weak convergence.
\end{theorem}

\begin{pf}
Suppose that $\mu_n \Rightarrow\mu$, and that
$\mu_n = \mu_n \circ\lambda^{(i)}$ for each $n$.
Let $F$ be any closed set in the product topology on $\Omega$. As
$\lambda^{(i)}$ is a continuous involution, $\lambda^{(i)}(F)$ is also
closed. Hence, we have
$\limsup\mu^n(\lambda^{(i)}(F)) \le\mu(\lambda^{(i)}(F))$, or in other
words
$\limsup(\mu_n \circ\lambda^{(i)})(F) \le(\mu\circ\lambda^{(i)})(F)$.
This shows that
$\mu_n = \mu_n\circ\lambda^{(i)} \Rightarrow\mu\circ\lambda^{(i)}$.
(Alternatively, we could appeal to the \textit{Continuous Mapping theorem}---see~(2.5), or Theorem~2.7, in Billingsley~\cite{billingsley}.)

As weak limits are unique when they exist, we now deduce that
$\mu= \mu\circ\lambda^{(i)}$.

Each of the $\cP_k$, and $\cP$, are intersections of sets of the form
$\{ \mu\dvtx \mu= \mu\circ\lambda^{(i)} \}$. Therefore, these too are
closed in the topology of weak convergence.
\end{pf}

%s7 ###
\section{Extremal order-invariant measures} \label{sec:eoim}

As we mentioned in Section~\ref{sec:oip}, it is immediate that any convex
combination of order-invariant measures is again an order-invariant
measure, so the family of all order-invariant measures is a convex subset
of the set of all causet measures. As usual, an \textit{extremal}
order-invariant measure is one that cannot be written as a non-trivial
convex combination of two others.

The family of all order-invariant measures is extremely rich. In this
and the next section, our aim is to show that all the \textit{extremal}
members of this family are of a very special form.

Under some very general conditions, the extremal measures are exactly
those that have ``trivial tail;'' see Bovier~\cite{Bovier} or
Georgii~\cite{Georgii}, for instance. We will prove a similar result
in our setting, regarding the tail $\sigma$-field $\cG$ introduced in the
previous section; it is possible to deduce this from the results in
Chapter~7 of Georgii---see in particular Theorems 7.7 and 7.12 and
Remark~7.13, as well as Section 7.2 therein---our proof is
self-contained, and also brings in a third equivalent condition that is of
special interest in our setting.

To explain this condition, we shall first show that, for each fixed\vspace*{1pt}
\mbox{$E \in\cF$}, the family $\nu^k(E)(\cdot)$ of random variables converges
a.s.\ to a limit $\cG$-measurable random measure $\nu(E)$.

\begin{theorem} \label{thm:nu}
Let $\mu$ be any order-invariant measure, and let $E$ be any event in
$\cF$. Then the sequence $\nu^k(E)(\omega)$ converges to a
$\cG$-measurable limit $\nu(E)(\omega)$ $\mu$-a.s. Moreover,
$\nu(E) = \mu(E \mid\cG)$, $\mu$-a.s., and $\E_\mu\nu(E) = \mu(E)$.
\end{theorem}

\begin{pf}
Firstly, as $\cG_1 \supseteq\cG_2 \supseteq\cdots,$ the sequence
$\mu(E \mid\cG_k) = \E_\mu(\bone_E \mid\cG_k)$ forms a backward
martingale with respect to the sequence $(\cG_k)$, for any $E \in\cF$.
Therefore, by the Backward Martingale Convergence theorem (see, e.g., Grimmett and Stirzaker~\cite{GrSt}),
$\mu(E \mid\cG_k) = \nu^k(E)$ converges a.s.\ to some random variable
$\nu(E)$. Given $n \in\N$, $\nu^k(E)$ is $\cG_n$-measurable for
$k \ge n$, and hence so is $\nu(E)$; therefore, the limit $\nu(E)$ is
$\cG$-measurable.

To show that $\nu(E) = \mu(E \mid\cG)$ a.s., we need to verify
that, for all $H \in\cG$,
%
%e7 ###
\begin{equation} \label{eq:three}
\E_\mu(\bone_{H \cap E}) = \E_\mu(\bone_H \nu(E)).
\end{equation}
By~(\ref{eq:nu-k}), we have that
\[
\E_\mu(\bone_{H \cap E}) = \E_\mu(\bone_H \nu^k(E))
\]
for every $H \in\cG$ and every positive integer $k$. By almost sure and
bounded convergence, the right-hand side of the last equation tends to
$\E_\mu(\bone_H \nu(E))$ as $k \to\infty$, as required.

The result that $\E_\mu\nu(E) = \mu(E)$ is the special case of
(\ref{eq:three}) with $H = \Omega$.
\end{pf}

We say that an order-invariant measure $\mu$ is \textit{essential} if, for
every $E \in\cF$, $\nu^k(E) \to\mu(E)$ a.s. In other words, $\mu$ is
essential if, for every $E$, the limit $\nu(E)$ in Theorem~\ref
{thm:nu} is
a.s.\ equal to $\mu(E)$, or equivalently $\mu(E \mid\cG) = \mu(E)$ a.s.

We say that $\mu$ has \textit{trivial tails} if $\mu(H)$ is equal to~0 or~1
for every $H \in\cG$. As usual, this is equivalent to saying that every
$\cG$-measurable random variable is constant. We now establish two
alternative characterizations of extremal order-invariant measures.

\begin{theorem} \label{thm:equivalence}
The following are equivalent for an order-invariant measure~$\mu$:
\begin{longlist}[(iii)]
\item[(i)]
$\mu$ is extremal;
\item[(ii)]
$\mu$ has trivial tails;
\item[(iii)]
$\mu$ is essential.
\end{longlist}
\end{theorem}

\begin{pf}
We will show that (ii) $\Rightarrow$ (iii) $\Rightarrow$ (i)
$\Rightarrow$ (ii).

(ii) $\Rightarrow$ (iii). If $\mu$ has trivial tails, then, for any
$E\in\cF$, the $\cG$-measurable random variable $\nu(E) = \mu
(E\mid\cG)$
is a.s.\ constant. As $\nu(E)$ is bounded, it is a.s.\ equal to its
expectation $\mu(E)$. Thus, $\mu$ is essential.

(iii) $\Rightarrow$ (i). Suppose $\mu$ is essential; we aim to show
it is
extremal.

If $\mu$ is not extremal, then we can write
$\mu= \alpha\mu_0 + (1-\alpha) \mu_1$ for two distinct order-invariant
measures $\mu_0,\mu_1$ and some $\alpha\in(0,1)$.

As $\mu_0 \not= \mu$, we may choose some $E \in\cF$ such that
$\mu_0(E) < \mu(E)$.

Applying Theorem~\ref{thm:nu} to $\mu_0$, for this event $E$, we obtain
that there is a~limiting random variable $\nu_0(E)$, such that
$\E_{\mu_0}(\nu_0(E)) = \mu_0(E)$ and $\nu^k(E) \to\nu_0(E)$,
$\mu_0$-a.s.

Then
$\mu_0(E) = \E_{\mu_0}(\nu_0(E)) \ge\mu(E)
\mu_0(\{ \omega\dvtx \nu_0(E)(\omega) \ge\mu(E)\})$, so
\[
\mu_0 \bigl(\{\omega\dvtx \nu_0(E)(\omega) \ge\mu(E)\}\bigr) \le
\frac{\mu_0(E)}{\mu(E)}< 1,
\]
which implies that $\mu_0 (\{\omega\dvtx \nu^k(E)(\omega) \to\mu(E)\}
) < 1$,
and so $\mu(\{\omega\dvtx \nu^k(E)(\omega) \to\mu(E)\}) < 1$. Thus,
$\mu$ is
not essential.

(i) $\Rightarrow$ (ii).
Suppose $\mu$ does not have trivial tails, and take some tail event
$H \in\cG$ with $0 < \mu(H) < 1$.

We now consider conditioning $\mu$ on the occurrence or not of $H$. Let
$\mu_1(E) = \mu(E\cap H)/\mu(H)$, and $\mu_0(E)=\mu(E\cap H^c)/\mu(H^c)$,
where $H^c$ is the complement of $H$, for every $E \in\cF$. Then
certainly $\mu= \mu(H) \mu_1 + (1 - \mu(H)) \mu_0$, and
$\mu_1 \not= \mu_0$. It remains to verify that $\mu_1$ and $\mu_0$ are
order-invariant: this will imply that $\mu$ is a convex combination of two
distinct order-invariant measures, so is not extremal.

It will suffice to consider $\mu_1$. We have to show that
\[
\mu\bigl(E\bigl(B_1\cdots B_k,<^{[k]}\bigr) \cap H\bigr) =
\mu\bigl(E\bigl(B_{\lambda(1)}\cdots B_{\lambda(k)},\lambda\bigl[<^{[k]}\bigr]\bigr) \cap H\bigr),
\]
whenever $B_1, \ldots, B_k$ are Borel sets, and $\lambda$ is a linear
extension of $<^{[k]}$.

By Theorem~\ref{thm:nu-k}, since $H \in\cG_k$ we have
\begin{eqnarray*}
\mu\bigl(E\bigl(B_1\cdots B_k,<^{[k]}\bigr)\cap H\bigr) &=&
\int_H \bone_{E(B_1\cdots B_k,<^{[k]})} \, d\mu(\omega) \\
&=&\int_H \nu^k\bigl(E\bigl(B_1\cdots B_k,<^{[k]}\bigr)\bigr)(\omega) \, d\mu(\omega).
\end{eqnarray*}
Lemma~\ref{lem:lambda-k-invariant} tells us that
\[
\nu^k\bigl(E\bigl(B_1\cdots B_k,<^{[k]}\bigr)\bigr)(\omega) =
\nu^k\bigl(E\bigl(B_{\lambda(1)}\cdots B_{\lambda(k)},\lambda
\bigl[<^{[k]}\bigr]\bigr)\bigr)(\omega)
\]
for all $\omega$; integrating over $H$ now implies that
\begin{eqnarray*}
&&\int_H \nu^k\bigl(E\bigl(B_1\cdots B_k,<^{[k]}\bigr)\bigr)(\omega) \, d\mu(\omega)\\
&&\qquad =
\int_H \nu^k\bigl(E\bigl(B_{\lambda(1)}\cdots B_{\lambda(k)},\lambda
\bigl[<^{[k]}\bigr]\bigr)\bigr)(\omega)
\, d\mu(\omega)
\\
&&\qquad = \mu\bigl(\bigl(E\bigl(B_{\lambda(1)}\cdots B_{\lambda(k)},\lambda\bigl[<^{[k]}\bigr]\bigr)\bigr) \cap H\bigr)
\end{eqnarray*}
again by Theorem~\ref{thm:nu-k}, which completes the proof.
\end{pf}

Our next result is somewhat related to Theorems~\ref{thm:nu}
and~\ref{thm:equivalence}: we show that, for any order-invariant measure
$\mu$, for $\mu$-a.e.\ $\omega$, the sequence $\nu^k(\cdot)(\omega
)$ of
measures converges weakly to a version of $\mu(\cdot\mid\cG)(\omega)$;
in particular, if $\mu$ is extremal, then, $\mu$-a.s., the
$\nu^k (\cdot)(\cdot)$ converge weakly to $\mu(\cdot)$. Although this
is superficially similar to the result that extremal order-invariant
measures are essential, the consequences are actually rather different.
Our proof is closely related to that of Proposition 7.25 in
Georgii~\cite{Georgii}.

\begin{theorem} \label{thm:weak}
There is a family $[\hat\nu(\cdot)(\omega)]_{\omega\in\Omega}$ of
order-invariant probability measures on $(\Omega, \cF)$ such that, for
any order-invariant measure $\mu$ on~$\cF$, and $\mu$-almost every
$\omega$,
\[
\nu^k(\cdot)(\omega) \Rightarrow\hat\nu(\cdot)(\omega).
\]

Moreover, for each fixed $E \in\cF$, and each order-invariant measure
$\mu$,
\[
\hat\nu(E)(\omega) = \mu(E \mid\cG)(\omega) = \nu(E)(\omega),
\qquad \mu\mbox{-a.s.}
\]
\end{theorem}

\begin{pf}
Since $(\Omega, \cF)$ is Borel, and complete and separable with
respect to
the product topology discussed in Section~\ref{sec:csp}, it is thus
\textit{standard Borel} in the terminology of Georgii~\cite{Georgii}. Then
by Theorem~(4.A11) in~\cite{Georgii}, $(\Omega, \cF)$ has a
\textit{countable core} $\cC$, that is, a countable collection of sets in
$\cF$
with the following properties:
\begin{longlist}[(ii)]
\item[(i)] $\cC$ generates $\cF$, and is a $\pi$-system,
\item[(ii)] whenever $(\nu^k)$ is a sequence of probability measures such
that $\nu^k(E)$ converges for all $E\in\cC$, then there is a (unique)
probability measure $\hat\nu$ on $(\Omega, \cF)$ such that
$\hat\nu(E) = \lim_{k\to\infty} \nu^k(E)$ for all $E \in\cC$.
\end{longlist}

Let $\cC$ be a countable core in $\cF$, and let
\[
\Omega_0 = \{ \omega\in\Omega\dvtx \nu^k(E)(\omega) \mbox{ converges }
\forall E \in\cC\}.
\]
For $\omega\in\Omega_0$ and $E\in\cC$, the limit
$\lim_{k \to\infty} \nu^k(E)(\omega)$ is equal to the $\cG$-measurable
function $\nu(E)(\omega)$, as defined in Theorem~\ref{thm:nu}: that
result tells us that $\mu(\Omega_0)=1$ for any order-invariant measure
$\mu$. We also note that $\Omega_0 \in\cF$, as the $\nu^k(E)$ are
$\cF$-measurable, and the statement that $(\nu^k(E)(\omega))$ is a
Cauchy sequence can be written in terms of these functions. Moreover,
$\Omega_0 \in\cG_k$ for all $k$, as $\nu^n(E)$ is $\cG_k$-measurable
for $n \ge k$, and so $\Omega_0 \in\cG$.

For each $\omega\in\Omega_0$, the $\nu^k(\cdot)(\omega)$ are probability
measures, and therefore, since $\cC$ is a core, the family
$\nu(E)(\omega)$ ($E \in\cC$) may be extended uniquely to a probability
measure $\hat\nu(\cdot)(\omega)$ on $(\Omega,\cF)$.

For convenience, we fix one order-invariant measure $\nu_0$, and set
$\hat\nu(\cdot)(\omega) = \nu_0(\cdot)$ for $\omega\notin\Omega_0$.

We next claim that, for each fixed $E\in\cF$, and any order-invariant
measure~$\mu$, $\hat\nu(E)$ is a version of $\mu(E\mid\cG)$.

Let $\cD= \{ E \in\cF\dvtx \hat\nu(E) \mbox{ is $\cG$-measurable}\}$.
The family $\cD$ contains the countable core $\cC$, since
$\hat\nu(E)$ coincides with $\nu(E)$ for $E\in\cC$---except on the
$\cG$-measurable set $\overline{\Omega_0}$, on which it is constant---and
we saw in Theorem~\ref{thm:nu} that $\nu(E)(\cdot)$ is $\cG$-measurable.
We also see that $\cD$ is a Dynkin-system (see, e.g., Williams~\cite{williams}), that is, it is closed under relative
complementation and increasing countable unions. By Dynkin's
$\pi$-$\lambda$ theorem (see~\cite{williams}, Theorem~A1.3), $\cD$
contains the $\sigma$-field generated by $\cC$, which is $\cF$.
Therefore, $\hat\nu(E)$ is $\cG$-measurable for all $E \in\cF$.

We now need to show that, for any order-invariant measure $\mu$,
\[
\int_H \hat\nu(E)\, d\mu= \int_H \bone_E \, d\mu
\]
for all $H\in\cG$ and $E\in\cF$. This is satisfied if $\mu(H) = 0$.
For other $H$, we can divide by $\mu(H)$ and express the required identity
as
\[
\frac{\E_\mu(\bone_H \hat\nu(E))}{\mu(H)} =
\frac{\mu(E \cap H)}{\mu(H)}
\]
for all $E\in\cF$. We see that both sides of the above identity are
probability measures on $(\Omega, \cF)$: the left-hand side is countably
additive, and equal to~1 for $E=\Omega$, since $\hat\nu$ is a probability
measure, while the right-hand side is the probability measure conditional
on~$H$. Moreover the two measures agree on $\cC$, as we established in
Theorem~\ref{thm:nu}. Since $\cC$ is a $\pi$-system generating~$\cF$,
this implies that the two measures are equal (see, e.g., Lemma~1.6
in Williams~\cite{williams}).

Thus, for each fixed $E\in\cF$, and each order-invariant measure $\mu$,
$\hat\nu(E)$ is indeed a version of $\mu(E\mid\cG)$.

Combining this with Theorem~\ref{thm:nu} tells us
that, for each $E\in\cF$, and each order-invariant measure $\mu$,
%
%e8 ###
\begin{equation} \label{eq:hatnu}
\hat\nu(E)(\omega) = \mu(E\mid\cG)(\omega) = \nu(E)(\omega) =
\lim_{k\to\infty} \nu^k(E)(\omega), \qquad   \mu\mbox{-a.s.}
\end{equation}

Now let $\tilde{\cC}$ be the family of events of the form
$E(B_1\cdots B_n,<^{[n]})$, where the $B_i$ are open intervals in
$[0,1]$ with rational endpoints in $[0,1]$ (including half-open
intervals with endpoints~0 or~1). Note that $\tilde{\cC}$ is a
countable family, and a basis for the product topology on $\Omega$.
Further, $\tilde{\cC}$ is a $\pi$-system. By Theorem 2.2 in
Billingsley~\cite{billingsley} (see also Examples 1.2 and 2.4 therein),
for weak convergence of a sequence of probability measures to a
probability measure, it is enough to verify convergence on the sets in
$\tilde{\cC}$. Let
\[
{\tilde\Omega}_0 = \{ \omega\in\Omega_0\dvtx \nu^k (E) (\omega) \to
\hat\nu(E)(\omega) \mbox{ } \forall E \in\tilde{\cC}\}.
\]
By (\ref{eq:hatnu}) and the choice of $\tilde{\cC}$ to be countable, we
have that $\mu({\tilde\Omega}_0)=1$, for any order-invariant measure
$\mu$.

Now, for $\omega\in{\tilde\Omega}_0$, since $\hat\nu(\cdot
)(\omega)$ and
all of the $\nu^k(\cdot)(\omega)$ are probability measures, and
$\nu^k(E)(\omega) \to\hat\nu(E)(\omega)$ for all $E \in\tilde
{\cC}$, we
deduce that $\nu^k(\cdot)(\omega) \Rightarrow\hat\nu(\cdot
)(\omega)$.

For each fixed $\ell$, all the measures $\nu^k(\cdot)(\omega)$ for
$k \ge\ell$ are $\Lambda^\ell$-invariant,\vspace*{1pt} by
Theorem~\ref{thm:lambda-k-invariant}. By Theorem~\ref{thm:closed}, it
follows that, for $\omega\in{\tilde\Omega}_0$, the weak limit
$\hat\nu(\cdot)(\omega)$ is $\Lambda^\ell$-invariant for each
$\ell$, and
so is order-invariant.
\end{pf}

Using this result, we obtain a fourth equivalent condition for an
order-invariant measure $\mu$ to be extremal. This condition is a weak
version of the property of being essential, which can be easier to
check, as we shall see shortly.

\begin{corollary} \label{cor:basic}
Let $\cH$ be the family of basic events $E = E(B_1\cdots B_n,\break <^{[n]})$
where each $B_i$ is a closed interval with rational endpoints.
Suppose that~$\mu$ is an order-invariant measure such that, for each
$E \in\cH$, $\nu^k(E) \to\mu(E)$ a.s. Then $\mu$ is essential, and
therefore extremal.
\end{corollary}

\begin{pf}
The property we need of $\cH$ is that it is a countable separating class.

Let $[\hat\nu(\cdot)(\omega)]_{\omega\in\Omega}$ be the family of
order-invariant measures guaranteed by Theorem~\ref{thm:weak}.
For each $E \in\cH$, we have that
$\hat\nu(E)(\omega) = \nu(E)(\omega) = \mu(E)$ for $\mu$-almost every
$\omega$. As $\cH$ is countable, this implies that, a.s.,
$\hat\nu(E) = \mu(E)$ for all $E \in\cH$. Since $\cH$ is a separating
class, and $\hat\nu$ and $\mu$ are both measures, this implies that
$\hat\nu= \mu$ a.s. Now, for any $E \in\cF$, an application of
Theorem~\ref{thm:weak} gives that $\nu(E) = \hat\nu(E) = \mu(E)$
a.s., so
$\mu$ is essential, as claimed.
\end{pf}

To illustrate some of the subtleties involved here, we consider processes
where the partial order $<^\N$ generated is a.s.\ an antichain, as in
Example~\ref{ex3}. Fix any element $\omega=(x_1x_2\cdots, <^\N)$ of $\Omega$,
where $<^\N$ is the antichain on $\N$, and $x_1,x_2,\ldots$ is a sequence
of distinct elements of $[0,1]$.

For a Borel subset $B$, $E(B) = E(B,<^\N_{[1]})$ is the event that the\vspace*{-2pt}
first element is in $B$. Now, for any $k$, $\nu^k(E(B))(\omega)$ is the
proportion of the elements $x_1, \ldots, x_k$ that lie in $B$. So
$\nu^k(E(\{x_j\}))(\omega)=1/k \to0$ as $k \to\infty$, for each fixed
$j$, yet $\nu^k(E(\{x_1,x_2, \ldots\}))(\omega) = 1$ for all $k$.

So, if we have a process that generates an antichain a.s., then we can
never have a measure $\hat\nu$ such that
$\nu^k(E)(\omega) \to\hat\nu(E)$ for \textit{every} set $E \in\cF$.
However, such sequences $\nu^k(\cdot)(\omega)$ may have weak limits.
Indeed, weak convergence to a measure $\hat\nu$ only guarantees
convergence on $\hat\nu$-continuity sets $E$, that is, sets $E$ whose
boundary $\partial E$ satisfies $\hat\nu(\partial E)=0$.

To be specific, consider the process that assigns independent uniform
labels from $[0,1]$ to the elements as they are generated. Then, for any
Borel subset $B$ of $[0,1]$, $\nu^k(E(B))$ is the proportion of elements
of $B$ among $X_k = \{x_1, \ldots, x_k\}$, and, by the strong law of large
numbers, $\nu^k(E(B)) \to|B|$ a.s., where $|\cdot|$ denotes Lebesgue
measure. Similarly, for $k\ge n$ and $<^{[n]}$ the antichain on $[n]$,
$\nu^k(E(B_1\cdots B_n, <^{[n]}))$ is the proportion of $n$-tuples of
distinct elements $(x_{i_1},\ldots,x_{i_n})$ from the set $X_k$ such that
$x_{i_j} \in B_j$ for each $j=1, \ldots, n$. This proportion tends to
$|B_1|\cdots|B_n|$ a.s.\ (and the limit is equal to~0 if $<^{[n]}$ is not
the antichain). The sequence $\nu^k(\cdot)$ thus a.s.\ converges weakly\vspace*{1pt}
to the product Lebesgue measure on $[0,1]^\N$. The process described here
is essential, and therefore extremal, by Corollary~\ref{cor:basic}.

This phenomenon can also be seen in otherwise well-behaved examples.
For instance, in Example~\ref{ex1}, where we have an order-invariant measure
on the fixed causal set $P$, let $\omega= (a_1a_2 \cdots)$---as before,
the order is implied---and let $E$ be the event
$\{ \omega=(x_1x_2\cdots) \in\Omega\dvtx x_i =a_i \mbox{ for all but
finitely many } i\}$. Then $\nu^k(E)(\omega) = 1$ for all $k$; however
$\nu^k(\cdot)(\omega) \Rightarrow\mu$, and $\mu(E) = 0$.

In Theorems~\ref{thm:nu} and~\ref{thm:weak}, we showed that, for each
fixed $E\in\cF$, and any order-invariant measure $\mu$,
$\nu^k(E)(\omega)$ tends $\mu$-a.s.\ to $\hat\nu(E)(\omega)$, where
$\hat\nu(\cdot)(\omega)$ is $\mu$-a.s.\ an order-invariant
measure. We
will now show that the measures $\hat\nu(\cdot)(\omega)$ are $\mu
$-a.s.\
extremal; it will follow that an order-invariant measure $\mu$ can be
decomposed uniquely as a mixture of these extremal order-invariant
measures.

Similar results are proved in Chapter~7 of Georgii~\cite{Georgii}.
Instead of using these results, we shall apply a result of Berti and
Rigo~\cite{berti-rigo} giving a ``conditional 0--1 law.''

The function taking $\omega\in\Omega$ to $\hat\nu(\cdot)(\omega
)$ is a
\textit{regular conditional distribution} for any order-invariant measure
$\mu$ given $\cG$: that is, each $\hat\nu(\cdot)(\omega)$ is a~%
probability measure, and that $\hat\nu(E) = \mu(E \mid\cG)$ $\mu$-a.s.,
for all $E \in\cF$---which we showed in Theorem~\ref{thm:weak}.
Moreover, as $\hat\nu(E)(\omega)$ is independent of the particular
order-invariant measure $\mu$, the tail $\sigma$-field $\cG$ is
\textit{sufficient} for the collection~$\cP$ of order-invariant measures.

The following result is (essentially) Lemma~5 of~\cite{berti-rigo}, which
in turn is adapted from a result of Maitra~\cite{maitra}.

\begin{lemma}[(Berti--Rigo)] \label{lem:berti-rigo}
Let $\cF$ be a countably-generated $\sigma$-field of subsets of a set
$\Omega$. Let $\Lambda$ be a countable set of $\cF$-measurable functions,
and let $\cP$ be the family of $\Lambda$-invariant probability
measures on
$(\Omega, \cF)$. Let $\cG$ be a sub-$\sigma$-field of $\cF$ that is
sufficient for $\cP$, and let $\hat\nu(\cdot)(\cdot)$ be a regular
conditional distribution for all $\mu\in\cP$ given $\cG$.

Then, for any $\mu\in\cP$, there is a set $G \in\cG$ with $\mu(G)
= 1$
such that $\hat\nu(H)(\omega) \in\{ 0,1\}$ for all $H \in\cG$ and
$\omega\in G$.
\end{lemma}

Evidently all the conditions of this lemma are satisfied in our setting,
and so the conclusion holds: it says that the probability measure
$\hat\nu(\cdot)(\omega)$ is tail-trivial, for $\mu$-almost every
$\omega$.

%There is a family $[\tilde\nu(\cdot)(\omega)]_{\omega\in\Omega}$ of
%extremal order-invariant probability measures on $(\Omega, \cF)$ such
%that, for any order-invariant measure $\mu$ on $\cF$, and $\mu$-almost
%every~$\omega$,
%$$
%$$
%
%Moreover, for each fixed $E \in\cF$, and each order-invariant measure
%$\mu$,
%$$
%$$
%
%Let $\nu_1$ be some particular extremal order-invariant measure, and
%define
%$$
%$$
%Note that $\Omega_1 = \{ \omega\in\Omega: \hat\nu(\cdot)(\omega)
%order-invariant measure $\mu$, by Lemma~\ref{lem:berti-rigo}. The
%result
%now follows from Theorem~\ref{thm:weak}.
%
%{\color{red} I'm not so sure about the proof above now. The question
%is:
%is the set $\Omega_1$ above in $\cF$ (if so, it is in $\cG$). I was
%hoping to replace $\Omega_1$ by the set $G$ in Lemma~
%but the problem with that is that $G$ depends on $\mu$.}

%We can use this corollary to prove the following decomposition result.

We now show that every order-invariant measure can be written uniquely as
a mixture of extremal order-invariant measures.

\begin{corollary} \label{cor:decomposition}
For any order-invariant measure $\mu$, there is a family
$[\tilde\nu(\cdot)(\omega)]_{\omega\in\Omega}$ of extremal
order-invariant probability measures on $(\Omega, \cF)$, with
$\tilde\nu= \hat\nu$, $\mu$-a.s., such that $\mu$ can be
decomposed as
%
%e9 ###
\begin{equation}
\label{eq:decomposition}
\mu(\cdot) = \int\tilde\nu(\cdot)(\omega)\, d\mu(\omega).
\end{equation}

Moreover, this is the unique decomposition of $\mu$ as a mixture of
extremal order-invariant measures, up to a.s.
\end{corollary}

\begin{pf}
Given an order-invariant measure $\mu$, let $G \in\cG$ be the set
guaranteed in Lemma~\ref{lem:berti-rigo}, with $\mu(G) = 1$ and
$\hat\nu(H)(\omega) \in\{0,1\}$ for all $H \in\cG$ and $\omega
\in G$.
For $\omega\in G$, $\hat\nu(\cdot)(\omega)$ is an extremal
order-invariant measure, by Theorem~\ref{thm:equivalence}.

Now let $\nu_1$ be some particular extremal order-invariant measure, and
define
\[
\tilde\nu(\cdot)(\omega) =
\cases{
\hat\nu(\cdot)(\omega), &\quad if $ \omega\in G$, \cr
\nu_1(\cdot), &\quad otherwise.
}
\]

By properties of conditional expectation, we have that, for all
$E \in\cF$, $\mu(E) = \E_\mu(\mu(E\mid\cG))$, which means that
\[
\mu(E) = \int\mu(E \mid\cG)(\omega) \, d\mu(\omega) =
\int\hat\nu(E)(\omega) \, d\mu(\omega)
\]
by Theorem~\ref{thm:weak}. We can write this in terms of the
measures as
\[
\mu(\cdot) = \int\hat\nu(\cdot) (\omega) \, d\mu(\omega).
\]
Since $\hat\nu(\cdot)(\omega) = \tilde\nu(\cdot)(\omega)$ for
$\mu$-almost every $\omega$, we also have the stated decomposition of
$\mu$, solely in terms of the extremal order-invariant measures
$\tilde\nu(\cdot)(\omega)$.

For uniqueness, we remark that $\tilde\nu(\cdot)(\cdot)$ is a
$(\cP,\cG)$-kernel, as in Definition~7.21 in~\cite{Georgii}.
Let $\cP_\cG$ be the family of tail-trivial (equivalently, extremal)
order-invariant measures. By Proposition~7.22 (or by Proposition~7.25,
Theorem~7.26 and comments at the end of Section~6.3) in
Georgii~\cite{Georgii}, there is a unique measure $w$ on the set
$\cP_\cG$, with the evaluation $\sigma$-field, such that
\[
\mu= \int_{\cP_\cG} \nu  w(d\nu)
\]
with $w$ given by $w(M) = \mu(\nu\in M)$, for $M$ a set in the
evaluation $\sigma$-field. Therefore, this must be the decomposition
in~(\ref{eq:decomposition}), as required.
\end{pf}

Alternatively, the result above can be deduced from the main result of
Maitra~\cite{maitra}, since $(\Omega,\cF)$ is a \textit{perfect space}.

As an illustration of all the ideas above, we return to Example~\ref{ex4}, where
we studied order-invariant measures on the poset $P$ consisting of two
chains. It is not too hard to see that the order-invariant measures
$\mu_q$, for fixed $q \in[0,1]$, are extremal; moreover, these are the
only order-invariant measures on $P$. (This is proved in detail
in~\cite{BL2}.) We also described the order-invariant measures $\mu
_\rho$,
where $\rho$ is a probability measure on $[0,1]$. These are, by
definition, mixtures of the $\mu_q$: they can be written as
\[
\mu_\rho(\cdot) = \int\mu_q(\cdot) \, d\rho(q).
\]
Corollary~\ref{cor:decomposition} now states that every order-invariant
measure on $P$ can be expressed as $\mu_\rho$, for some probability
measure $\rho$ on $[0,1]$.

Given a causal set $P=(Z,<)$, and an element
$\omega= x_1x_2\cdots\in\Omega$, we say that $\omega$ \textit{generates}
a measure $\mu$ on $(\Omega_P, \cF)$ if $\nu^k(\cdot)(\omega)$ converges
weakly to $\mu$ as $k\to\infty$.

Theorem~\ref{thm:weak} tells us that, if $\mu$ is extremal, then
$\nu^k(\cdot)(\omega)$ converges weakly to $\mu$ a.s.: in other words,
$\mu$-almost all $\omega$ generate $\mu$. In particular, if $\mu$ is
extremal, then $\mu$ is generated by at least one $\omega\in\Omega$.
We suspect that the converse is likely to be true.

\begin{conjecture}
If $\mu$ is an order-invariant measure that is generated by some
$\omega\in\Omega$, then $\mu$ is extremal.
\end{conjecture}

The best result we can prove in this direction is the following.

\begin{theorem}
For an order-invariant measure $\mu$, let
$\Omega_0 = \{\omega\dvtx \break\omega\mbox{ generates}~\mu\}$ and suppose
that $\mu(\Omega_0) > 0$. Then $\mu$ is extremal.
\end{theorem}

\begin{pf}
Consider the family $\tilde\nu(\cdot)(\omega)$ of extremal order-invariant
measures in Corollary~\ref{cor:decomposition}. Set
\[
{\tilde\Omega}_0 = \{ \omega\in\Omega_0 \dvtx \nu^k(\cdot)(\omega)
\Rightarrow\tilde\nu(\cdot)(\omega)\}.
\]
Then $\mu({\tilde\Omega}_0) = \mu(\Omega_0) > 0$, by
Theorem~\ref{thm:weak} and Corollary~\ref{cor:decomposition}. In
particular, ${\tilde\Omega}_0$ is non-empty; for any
$\omega\in{\tilde\Omega}_0$, $\mu$ is the weak limit of the
$\nu^k(\cdot)(\omega)$, and is therefore equal to the extremal
order-invariant measure $\tilde\nu(\cdot)(\omega)$.\vspace*{-2pt}
\end{pf}

%s8 ###
\section{Description of extremal order-invariant measures}
\label{sec:deoim}

Our aim in this section is to prove the following result.

\begin{theorem} \label{thm:extremal}
Let $\mu$ be an extremal order-invariant measure. Then there is a poset
$Q=(Z,<)$, either a causal set or a finite poset, with a~marked set $M$ of
maximal elements such that, if $Q'$ is obtained from $Q$ by replacing each
element $z$ of $M$ with a countably infinite antichain $A_z$, then the
poset $\Pi$ generated by $\mu$ is a.s.\ equal to $Q'$, except for the
labels on the antichains~$A_z$.
\end{theorem}

Probably the most interesting special case is when the set $M$ of marked
maximal elements is empty, so that the extremal measure $\mu$ is an
order-invariant measure on the fixed (labeled) causal set $Q$. As we saw
in Example~\ref{ex4}, and the discussion after Corollary~\ref{cor:decomposition},
not every order-invariant measure on a fixed causal set is extremal---indeed, whenever there is more than one order-invariant measure on a fixed
causal set $P$, a non-trivial convex combination will be non-extremal---so Theorem~\ref{thm:extremal} falls short of characterizing extremal
order-invariant measures.

In our companion paper~\cite{BL2}, we discuss at length the issue of which
fixed causal sets admit an order-invariant measure. We have seen examples
in this paper of causal sets that admit just one order-invariant measure
(Example~\ref{ex1}), many order-invariant measures (Example~\ref{ex4}), or none (e.g., a
labeled antichain: see Example~\ref{ex3}).

The other extreme case is when $Q$ consists of a single marked element
$z$, and so $Q'$ consists of the single antichain $A_z$. As discussed in
Example~\ref{ex3}, an order-invariant measure that a.s.\ generates an antichain is
effectively the same as an exchangeable sequence of random labels $[0,1]$,
and the extremal order-invariant measures correspond to the atomless
probability distributions on $[0,1]$.

Theorem~\ref{thm:extremal} allows intermediate cases as well. For
instance, suppose~$Q$ consists of a chain $y_1<y_2< \cdots$ of elements
from $[0,1]$, together with a~marked element $z_i$ above each $y_i$.
Suppose we are also given atomless probability distributions $W_i$ on
$[0,1]$ for each $i$, and a strictly decreasing sequence of positive real
numbers $1=p_1 > p_2 > \cdots$. Then the following causet process is
order-invariant. Suppose we are at a state in which the elements
$y_1, \ldots, y_{r-1}$ are present, but $y_r$ is not. Then, with
probability $p_r$, select $y_r$ and place it above $y_{r-1}$; for
$j=1, \ldots, r-1$, with probability $p_j-p_{j+1}$, select an element from
$[0,1]$ according to the distribution $W_j$, and place it above~$y_j$. It
is easily checked that this is an order-invariant measure, and indeed that
it is extremal.

%Theorem~\ref{thm:extremal} does not [yet] say that, if we do have a
%maximal element $z$ in $Q$, and a corresponding antichain $A_z$ in
%$Q'$,
%then the labels of $A_z$ are generated as an iid sequence of random
%elements from $[0,1]$. Nevertheless this is true.
%Return to this?

Before proving Theorem~\ref{thm:extremal}, we need a number of preliminary
results and definitions.

Our first tools are from the theory of linear extensions of finite posets.
For a finite poset $P=(Z,<)$, let $\nu^P$ denote\vadjust{\goodbreak} the uniform measure on
linear extensions of $P$. We shall denote a uniformly random linear
extension of an $n$-element poset by $\zeta= \zeta_1\cdots\zeta_n$, and
we shall set $\Sigma_i(\zeta) = \{ \zeta_1, \ldots, \zeta_i\}$, the set
consisting of the bottom $i$ elements of a uniformly random linear
extension $\zeta$ of a finite poset. (It is useful to have different
notation for uniformly random linear extensions of finite posets and for
random samples from order-invariant measures on causal sets, as we shall
shortly need to consider both notions simultaneously.)

For a finite poset $P=(Z,<)$, and $z_1,\ldots,z_k \in Z$, let
$E(z_1\cdots z_k)$ be the set of linear extensions of $P$ with initial
segment $z_1\cdots z_k$, so that $\nu^P(E(z_1\cdots z_k))$
is the proportion of linear extensions of $P$ with initial segment
$z_1\cdots z_k$. In particular, $\nu^P(E(z))$ is the probability that a
uniformly random linear extension of $P$ has $z$ as its bottom element.

\begin{lemma} \label{lem:corr}
Let $P=(Z,<)$ be a finite poset, and suppose that $x$ is a~minimal element
of~$P$. Let $D$ be a down-set in $P$, not including $x$. Then
$\nu^P(E(x)) \le\nu^{P\sm D}(E(x))$.
\end{lemma}

In other words, if $x$ is a minimal element of $P$, and the probability
that~$x$ is the bottom element of a uniformly random linear extension of
$P$ is $p$, then the probability that $x$ is the bottom element of a
uniformly random linear extension of $P\sm D$ is always at least~$p$, for
any down-set $D$ of $P$ not including $x$. Hopefully this seems
intuitively plausible, but, as is often the case with correlation
inequalities, no completely elementary proof is known.

Lemma~\ref{lem:corr} can be seen as a special case of the following
inequality, due to Fishburn~\cite{Fishburn}.

\begin{theorem} [(Fishburn)] \label{thm:fishburn}
Let $U$ and $V$ be up-sets in a finite poset $P=(Z,<)$, and, for
$Y\subseteq Z$, let $e(Y)$ denote the number of linear extensions of
$P_Y$. Then
\[
e(U) e(V) \le e(U \cup V) e(U \cap V).
\]
\end{theorem}

Indeed, setting $U = Z\sm\{x\}$ and $V = Z\sm D$, we have that
$\nu^P(E(x))=e(U)/e(U\cup V)$ and $\nu^{P\sm D}(E(x)) = e(U\cap V)/e(V)$;
Fishburn's inequality in this case is exactly Lemma~\ref{lem:corr}.

Theorem~\ref{thm:fishburn} was first proved by Fishburn
in~\cite{Fishburn}; Brightwell gave a simpler proof in~\cite{Bri1}. A
version of Lemma~\ref{lem:corr} is used as part of the proof of Lemma~3.5
in Brightwell, Felsner and Trotter~\cite{BFT}.

We make use of Lemma~\ref{lem:corr} in the proof of our next result, which
is the key to the proof of Theorem~\ref{thm:extremal}.

\begin{lemma} \label{lem:finite}
Let $P=(X,<)$ be a finite poset, and take $\delta> 0$ and $k \in\N$ such
that $k\delta\le1$. Suppose that $\cZ$ is a family of down-sets $Z$ in
$P$, each with $|Z|\le k$, such that $\varnothing\in\cZ$ and, whenever
$Z \in\cZ$, $|Z|\le k-1$, and $Z\cup\{x\} \notin\cZ$, we have
$\nu^{P\sm Z}(E(x)) \le\delta$.

Let $Y$ be the union of the sets in $\cZ$, and let $M$ be the set of
minimal elements of $P\sm Y$. Then
\[
\nu^P \bigl(\{\zeta\dvtx \Sigma_k(\zeta) \subseteq Y \cup M\}\bigr) \ge
\prod_{j=1}^k \bigl(1 - (j-1)\delta\bigr) \ge1 - \pmatrix{k\cr2} \delta.
\]
\end{lemma}

The idea is that $Y$ contains all elements of the poset that are
``likely'' to appear among the first $k$ elements, even conditioned on
other ``likely'' events. The conclusion states that, with high
probability, all of the first $k$ elements are either in $Y$ or are
minimal in $P \sm Y$---in other words, the first~$k$ elements do not
contain a pair of comparable elements that are not in $Y$.

Lemma~\ref{lem:finite} is asymptotically best possible, at least in the
case where $1/\delta= m \in\N$. To see this, let $P$ be the disjoint
union of $m$ chains, each of length $t$, with $\cZ= \{\varnothing\}$, so
$Y = \varnothing$ and $M$ consists of the bottom elements of the chains.
For $k\le m$, the probability that, in a uniformly random linear
extension of $P$, the bottom $k$ elements are all in $M$---that is,
all in  different chains---is asymptotically equal to the product above as $t$
tends to infinity.

\begin{pf*}{Proof of Lemma \ref{lem:finite}}
We call a down-set $D$ of $P$ \textit{low} if it is the union of a set
$Z \in\cZ$ and a set $W$ of minimal elements of $P\sm Z$. Note that each
low down-set is a subset of $Y\cup M$. If $D$ is a low down-set, we may
and shall take $Z$ to be a maximal element of $\cZ$ with $Z \subseteq D$,
and $W=D\sm Z$, so that $Z \cup\{w\} \notin\cZ$ for each $w \in W$.

Let $D= Z \cup W$ be a low down-set as above, with $|Z| \le k-1$. This
implies that $\nu^{P\sm Z}(E(x)) \le\delta$ for $x \in W$. Let $N =N(Z)$
denote the set of minimal elements of $P\sm Z$, so $W \subseteq N$, and
note that each set $D\cup\{x\}$, for $x \in N\sm W$, is a low down-set.
We claim that $\nu^{P\sm D}(\bigcup_{x \in N\sm W} E(x))$---the
probability that, in a uniformly random linear extension of $P\sm D$, the
bottom element $x$ is in $N\sm W$---is at least $1 - |W| \delta$.

We start by considering the probability that each element is bottom in
a~uniformly random linear extension of the larger poset $P\sm Z$. For
$x \in N$, set $p_x = \nu^{P\sm Z}(E(x))$.
%the probability that $x$ is the
%bottom element in a uniformly random linear extension of $P \sm Z$.
Note that $\sum_{x\in N} p_x =1$, and also that
$p_x = \nu^{P\sm Z}(E(x)) \le\delta$ for each $x\in W$.

We consider now the poset $P\sm D = (P\sm Z)\sm W$, and the various
probabilities that an element is bottom in a uniformly random linear
extension of this poset. For $x \in N\sm W$, we set
$q_x= \nu^{P\sm D}(E(x))$; by Lemma~\ref{lem:corr}, we have $q_x \ge p_x$
for all $x \in N\sm W$. Thus,
\[
\nu^{P\sm D}\Biggl(\bigcup_{x \in N\sm W}E(x)\Biggr) =
\sum_{x \in N\sm W} q_x \ge\sum_{x\in N\sm W} p_x =
1 - \sum_{x\in W} p_x \ge1- |W| \delta
\]
as claimed.\vadjust{\eject}

To complete the proof, observe that, for $1\le j \le k$,
\[
\nu^P(\Sigma_j \mbox{ is low } \mid\Sigma_{j-1} \mbox{ is low})
\]
is a convex combination of terms of the form
$\nu^P(\Sigma_j \mbox{ is low } \mid\Sigma_{j-1} = D = Z \cup W)$,
where $D=Z\cup W$ is a low down-set of size $j-1$. As all the down-sets
$D\cup\{x\}$, for $x \in N(Z) \sm W$, are low, we have
\begin{eqnarray*}
\nu^P(\Sigma_j \mbox{ is low } \mid\Sigma_{j-1} = D = Z \cup W)
&\ge&
\nu^{P\sm D}\Biggl(\bigcup_{x\in N(Z)\sm W} E(x)\Biggr) \\
&\ge&1 - |W|\delta\ge
1-(j-1)\delta.
\end{eqnarray*}
Therefore,
\[
\nu^P(\Sigma_j \mbox{ is low } \mid\Sigma_{j-1} \mbox{ is low})
\ge1 -(j-1)\delta
\]
for each $j$. Multiplying terms, we see that
\[
\nu^P(\Sigma_k \mbox{ is low}) \ge\prod_{j=1}^k \bigl( 1-(j-1)\delta\bigr).
\]
The result follows.
\end{pf*}

Next, we state a result of Stanley~\cite{Stanley}. For an element $x$
in a
finite poset $P=(Z,<)$, let $r_i(x) = \nu^P(\{\zeta\dvtx \zeta_i = x\}
)$, the
probability that, in a uniformly random linear extension $\zeta$ of $P$,
$x$ appears in position~$i$.

\begin{theorem} [(Stanley)] \label{thm:stanley}
For any element $x$ in an $n$-element poset $P=(Z,<)$, the sequence
$(r_i(x))_{i=1}^n$ is log-concave.
\end{theorem}

There are many equivalent ways of expressing the property of
log-concavity. One is that the sequence of ratios $r_{i+1}(x)/r_i(x)$ is
nonincreasing over the range of $i$ for which $r_i(x) > 0$. This implies
that, for $j\le j+m \le j+s$, we have
\[
 \biggl(\frac{r_{j+s}(x)}{r_j(x)} \biggr)^{1/s} \!=\!
 \Biggl( \prod_{i=1}^s \frac{r_{j+i}(x)}{r_{j+i-1}(x)}  \Biggr)^{1/s}
\!\le\!
 \Biggl( \prod_{i=1}^m \frac{r_{j+i}(x)}{r_{j+i-1}(x)}  \Biggr)^{1/m}
\!=\!  \biggl(\frac{r_{j+m}(x)}{r_j(x)} \biggr)^{1/m}.
\]
This is the inequality we use to prove the following lemma.

\begin{lemma} \label{lem:q}
Fix $0<\varepsilon<1$, $0<\delta<1$ and $k\in\N$. Suppose $P=(Z,<)$ is a finite
poset. Let $L$ denote the set of elements $x$ of $P$ such that
$\nu^P(\{\zeta\dvtx x \in\Sigma_k(\zeta)\}) \ge\delta^{k+1}$.

Set $q=q(k,\delta,\eps)= 10k\delta^{-(k+1)}\log(5k/\eps\delta^{k+1})$.
Then
\[
\nu^P\bigl(\{ \zeta\dvtx L \subseteq\Sigma_q(\zeta)\}\bigr) > 1-\eps/8.
\]
\end{lemma}

Loosely: if $L$ is the set of elements that have a significant probability
of appearing within the bottom $k$ positions in a uniformly random linear
extension of a finite poset $P$, then, for sufficiently large $q$, it
is very
likely that all the elements of $L$ appear within the bottom $q$ positions.

\begin{pf*}{Proof of Lemma \ref{lem:q}}
Set $\eta= \delta^{k+1}$ for convenience. Note that the number $\ell
$ of
elements of $L$ is at most $k\eta^{-1}$. Now fix any element $x$ of
$L$, set
$r_i=r_i(x)$, for each $i$, and consider the sequence $(r_i)$.

By assumption, $\sum_{i=1}^k r_i \ge\eta$. So one of $r_1,\ldots,r_k$,
say $r_j$, is at least $\eta/k$. Also, as the $r_i$ sum to~1, one of the
next $\lceil2k/\eta\rceil$ terms
$r_{j+1},\ldots,r_{j+\lceil2k/\eta\rceil}$, say $r_{j+m}$, is at most
$\eta/2k$. By Theorem~\ref{thm:stanley}, the sequence $(r_i)$ is
log-concave. So, for $s \ge m$,
\[
 \biggl( \frac{r_{j+s}}{r_j} \biggr)^{1/s} \le
 \biggl(\frac{r_{j+m}}{r_j} \biggr)^{1/m} \le2^{-1/m} \le
2^{-\eta/3k}.
\]
Therefore, for $t\ge m$,
\[
\sum_{s=t}^\infty r_{j+s} \le\sum_{s=t}^\infty2^{-s\eta/3k} =
\frac{2^{-t\eta/3k}}{1-2^{-\eta/3k}} \le\frac{3k}{\eta}
2^{-t\eta/3k}.
\]

This implies that, for $t \ge2k/\eta$, the probability that a particular
element of $L$ is not among the bottom $k+t$ elements
$\xi_1, \ldots, \xi_{k+t}$ is at most $3k\eta^{-1}2^{-t\eta/3k}$. The
probability that some element of $L$ is not among the bottom $k+t$\vspace*{1pt} is thus
at most $3k^2\eta^{-2}2^{-t\eta/3k}$. Provided
$t \ge3k\eta^{-1}\log_2(24k^2/\eps\eta^2)$, this probability is at most
$\eps/8$. We set $t=q-k$, where $q$
is as in the statement of the lemma.  Noting that $t$ is large enough, we
are done.
\end{pf*}

Next, we establish some properties of an extremal order-invariant
measure~$\mu$. In what follows, we make heavy use of
Theorem~\ref{thm:equivalence}, which tells us that~$\mu$ is
essential, and
that~$\mu$ has trivial tails. We shall also use Theorem~\ref{thm:weak},
which tells us that the sequence $\nu^k(\cdot)(\omega)$ a.s.\ converges
weakly to $\mu$.

For $z \in[0,1]$ and $k \in\N$, we define the event
$G(z,k) = \{\omega' \in\Omega\dvtx z \in\Xi_k(\omega')\}$: this means that
$z$ is one of the first $k$ elements generated. For a fixed
$\omega=(x_1x_2\cdots, <^\N) \in\Omega$, observe that
$\nu^n(G(z,k))(\omega)$ is equal to
$\nu^{P_n}(\{ \zeta\dvtx z \in\Sigma_k(\zeta)\})$, the probability
that $z$
appears among the bottom $k$ elements in a uniformly random linear
extension $\zeta$ of the finite poset $P_n = \Pi_n(\omega)$. Indeed, for
any event $G \in\cF_n$, we have these two different interpretations of
$\nu^n(G)(\omega)$, and it is usually convenient to work with the latter.

\begin{lemma} \label{lem:gzk}
Let $\mu$ be an extremal order-invariant measure. For $\mu$-almost every
$\omega\in\Omega$, $\nu^n(G(z,k))(\omega) \to\mu(G(z,k))$ for every
$z \in[0,1]$ and $k\in\N$.
\end{lemma}

\begin{pf}
$\!\!$Theorem~\ref{thm:weak} tells us that, $\mu$-a.s., $\nu^n(\cdot
)(\omega)$
converges weakly to~$\mu$. For an $\omega$ such that weak convergence
holds, this implies that $\nu^n(E)(\omega) \to\mu(E)$ for all events
$E\in\cF$ such that $\mu(\partial E) = 0$, where $\partial E$
denotes the
boundary of $E$.

Each of the events $G(z,k)$ is a union of finitely many sets of the form
$\{ \omega\in\Omega\dvtx \xi_i(\omega) = z\}$, each of which is a
closed set
with empty interior. Thus, $G(z,k)$ is itself a closed set with empty
interior, so the boundary $\partial G(z,k)$ is the event $G(z,k)$ itself.

Next, we note that, for each $k\in\N$ and $m\in\N$, there are at most
$km$ elements $z\in[0,1]$ such that $\mu(G(z,k))\ge1/m$. Therefore,
the set $C$ of pairs $(z,k)$ such that $\mu(G(z,k)) > 0$ is countable.

As $\mu$ is essential, we have, $\mu$-a.s., that
$\nu^n(G(z,k))(\omega) \to\mu(G(z,k))$ for all $(z,k)$ in the countable
set $C$.

Let
\[
\Omega_0 = \{ \omega\in\Omega\dvtx \nu^n(\cdot)(\omega) \Rightarrow
\mu(\cdot) \mbox{ and } \nu^n(G(z,k))(\omega) \to\mu(G(z,k))\
\forall
(z,k) \in C \}.
\]
We have that $\mu(\Omega_0) = 1$.

Now fix $\omega\in\Omega_0$. For $(z,k) \in C$, we know that
$\nu^n(G(z,k))(\omega) \to\mu(G(z,k))$. On the other hand, for every
$(z,k) \notin C$, we have that $\mu(\partial G(z,k)) = \break\mu(G(z,k)) = 0$;
as $\nu^n(\cdot)(\omega)$ converges weakly to $\mu$, this implies
that\break
$\nu^n(G(z,k))(\omega) \to\mu(G(z,k)) = 0$ for all $(z,k) \notin C$.
Hence, $\nu^n(G(z,k))(\omega) \to\mu(G(z,k))$ for all pairs $(z,k)$,
as required.
\end{pf}

Let $\mu$ be an extremal order-invariant measure. For each $z \in[0,1]$,
the event $G(z)= \{ \omega\dvtx z \in\Xi(\omega)\}$---the event that $z$
is generated at all---is a tail event, so $\mu(G(z))$ is either~0 or~1.
Set
\[
V = V(\mu) = \{z \in[0,1]\dvtx \mu(G(z)) = 1\}.
\]
Referring to the statement of Theorem~\ref{thm:extremal}, one of our goals
is to identify $V$ with $Q\sm M$.

We say that the element $z \in[0,1]$ is \textit{persistent} for
$\omega\in\Omega$ if, for some $k \in\N$,
\[
\liminf_{n\to\infty} \nu^n(G(z,k))(\omega) > 0.
\]
Lemma~\ref{lem:gzk} tells us that, a.s., all the limits
$\lim_{n\to\infty} \nu^n(G(z,k))(\omega)$ exist, and are equal to the
corresponding $\mu(G(z,k))$. Thus, a.s., the elements that are persistent
for $\omega$ are exactly those with $\mu(G(z,k)) > 0$ for some $k$,
which in turn are exactly those in $V$.

Notice that, for $\omega= (x_1x_2\cdots, <^{\N})$, only elements
appearing in the string $x_1x_2 \cdots$ can be persistent for $\omega$.
However, elements that do appear in the string need not be persistent.
Consider, for instance, any element\vspace*{1pt}
$\omega=\break (x_1x_2\cdots, <^\N) \in\Omega$, where $<^\N$ is an antichain.
Here, we have $\nu^n(G(x_1,k))(\omega) = k/n$ whenever $n \ge k$, so $x_1$
is not persistent for $\omega$, and indeed no element is persistent for
such an $\omega$. In the setting of Example~\ref{ex3}, where the generated
partial order is a.s.\ an antichain, this means that there are a.s.\ no
persistent elements.

We know that, for $\mu$-almost every $\omega$, for each element
$z \notin V$, and for each $k \in\N$, $\nu^n(G(z,k))(\omega) \to
0$. For
fixed $k$, we now want to establish the existence of a suitably large
$n_0$ so that, for all $\omega$ in some set with high $\mu$-probability,
all of the $\nu^{n_0}(G(z,k))(\omega)$, for $z\notin V$, are small.
Although the previous results do not give us any form of \textit{uniform}
convergence of the sequences $\nu^n(G(z,k))(\omega)$ for all $z
\notin V$,
the following result---covering only the elements $z\notin V$ that
appear in the set $\Xi_q(\omega)$ of the first $q$ elements generated---will be sufficient for our purposes.

\begin{lemma} \label{lem:quant}
Let $\mu$ be an extremal order-invariant measure. Fix $\eps> 0$,
$\delta>0$ and $k \in\N$, and let $q$ be as in Lemma~\ref{lem:q}. Then
there exists $n_0 \in\N$ such that, for all $n \ge n_0$,
\[
\mu\bigl( \{ \omega\dvtx \mbox{for all elements $z$ of $\Xi_q(\omega)\sm V$,
$\nu^n(G(z,k))(\omega) < \delta^{k+1}$} \}\bigr) > 1 -\eps/8.
\]
\end{lemma}

\begin{pf}
We have seen that, a.s., there are no elements persistent for $\omega$ other
than those in $V$. In particular, for each $j=1, \ldots, q$, we have
\[
\mu\bigl(\{ \omega\dvtx \xi_j(\omega) \notin V,  \xi_j(\omega)
\mbox{ is persistent for } \omega\}\bigr) = 0.
\]
Therefore, for $j=1, \ldots, q$, there is some $n_0(j)$ such that
\[
\mu\bigl(\{ \omega\dvtx \xi_j (\omega) \in V \mbox{ or }
\nu^n(G(\xi_j(\omega),k))(\omega) < \delta^{k+1} \mbox{ for all }
n\ge n_0(j)\}\bigr) \ge1 - \eps/8q.
\]
Choosing $n_0$ to be the maximum of the $n_0(j)$ now gives the desired
result.
\end{pf}

\begin{pf*}{Proof of Theorem~\ref{thm:extremal}}
Let $\mu$ be an extremal order-invariant measure on $(\Omega, \cF)$.

For a given $\omega\in\Omega$, the set of elements that are persistent
for $\omega$ forms a~down-set in the causal set $\Pi(\omega)$. We have
seen that this down-set is a.s.\ the set $V = V(\mu)$.

For any $x,y \in V$, the event that $x\prec y$ in $\Pi$ is a tail event.
Therefore, $\mu(\{ \omega\dvtx x\prec y \mbox{ in } \Pi(\omega)\})$ is equal
to~0 or~1. The relation $<^V$ on $V$ defined by $x<^V y$ if and only if
$\mu(\{ \omega\dvtx x \prec y \mbox{ in } \Pi(\omega)\}) = 1$ is a partial
order on $V$, and moreover the restriction $\Pi(\omega)_V$ is a.s.\ equal
to $(V,<^V)$.

Consider any down-set $D$ of $(V,<^V)$, and let $\Gamma_D$ be the (random)
set of elements $y$ of $\Xi\sm V$ such that the set of elements below
$y$ in $\Pi$ is exactly equal to $D$. Then $|\Gamma_D|$ is a random
variable taking values in $\N\cup\{0,\infty\}$. For each~$g$, the set
$\{ \omega\dvtx |\Gamma_D(\omega)| = g\}$ is a tail event, and so
$|\Gamma_D|$ is a.s.\ determined. Our next aim is to show that, for each
$D$, the a.s.\ value of $|\Gamma_D|$ is either $0$ or $\infty$.

Suppose, for a contradiction, that $|\Gamma_D|$ is a.s.\ equal to the
positive integer~$m$. For any Borel set $B \subseteq[0,1]$, the
event $K_D(B) = \{ \omega\dvtx \Gamma_D(\omega) \cap B \not=\varnothing
\}$ is
a~tail event, so $\mu(K_D(B))$ is equal to~0 or~1. We say that $B$ is
\textit{occupied} if $\mu(K_D(B)) =1$. No singleton set $\{v\}$ is occupied:
if $v \in\Gamma_D$ a.s., then certainly $v \in\Xi$ a.s., so $v\in V$,
but we have defined $\Gamma_D$ to be disjoint from $V$.

If $B$ is occupied, and $\{B_1, B_2, \ldots\}$ is any covering of $B$ with
countably many Borel sets $B_i$, then
$K_D(B) \subseteq\bigcup_{i=1}^\infty K_D(B_i)$, so at least one set
$B_i$ in the covering is occupied.

Let $B_1=[0,1]$ and note that $B_1$ is occupied. By repeatedly
interval-halving, we can find a decreasing sequence of closed intervals
$B_1 \supset B_2 \supset\cdots$, with $B_i$ of length $2^{1-i}$, each of
which is occupied, and whose intersection is a single point $v \in[0,1]$.
We now partition $[0,1]$ as the countable union of the Borel sets
$B_j \setminus B_{j+1}$, $j=1,2,\ldots,$ together with the singleton
$\{v\}$. We observe that at most $m$ of these sets are occupied:
otherwise there are a.s.\ more than $m$ elements of $[0,1]$ in $\Gamma_D$.
Moreover, the occupied sets do not include the singleton $\{ v\}$. Hence
there is a maximum $k$ such that $B_k \setminus B_{k+1}$ is occupied. But
then we have a partition of the occupied set $B_{k+1}$ into countably many
sets, none of which are occupied. This is a contradiction.

We conclude that, for each down-set $D$ of $(V,<^V)$, the set $\Gamma_D$
is either a.s.\ empty, or a.s.\ infinite.

We say that a down-set $D$ of $(V,<^V)$ is \textit{active} if $\Gamma_D$ is
a.s.\ infinite. One consequence of what we have just proved is that,
a.s., all minimal elements of the poset restricted to $\Xi\sm V$ have a
down-set (necessarily a subset of $V$) that is active.

We are now in a position to construct the causal set $Q=(Z,<)$ in the
statement of the theorem. We take the causal set $(V,<^V)$, and add a
marked maximal element $z_D$ above each active down-set $D$. We have
already seen that the random causal set $\Pi$ a.s.\ contains $(V,<^V)$ as
a down-set, and infinite antichains $\cA_D$ above each active
down-set~$D$. What remains to be shown is that there are a.s.\ no other
elements in $\Xi$: specifically, we have shown that, a.s., the set
$\cA$
of minimal elements of $\Pi\sm V$ is the union of the infinite antichains
$\cA_D$; we now need to show that there are a.s.\ no nonminimal elements
of $\Pi\sm V$.

We shall prove the following equivalent statement. For every $\eps>0$,
and every $k \in\N$,
%
%e10 ###
\begin{equation} \label{eq:eps-k}
\mu\bigl( \{ \omega\dvtx \Xi_k(\omega) \subseteq V \cup\cA(\omega) \}\bigr)
\ge
1- \eps,
\end{equation}
that is, the probability that the first $k$ elements include a pair of
comparable elements that are not in $V$ is at most $\eps$.

For the remainder of the proof, we fix a natural number $k\ge2$, and
some~$\eps$ with $0<\eps\le1$. We set $\delta=\eps/2k^2 >0$ and
$q = 10k\delta^{-(k+1)}\log(5k/\eps\delta^{k+1})$, as in
Lemma~\ref{lem:q}. We note here for future use that
$\delta^k < \eps/8$.

For a given string $z_1z_2 \cdots z_m$ of elements of $V$, let
$<^{[m]}$ be the order on~$[m]$ inducing $<^V_{\{z_1, \ldots, z_m\}}$,
that is, with $i<^{[m]}j$ if and only if $z_i <^V z_j$, and define
\[
E(z_1z_2 \cdots z_m) = E\bigl(\{z_1\}\{z_2\}\cdots\{z_m\}, <^{[m]}\bigr),
\]
the event that an element $\omega\in\Omega$ has $z_1\cdots z_m$ as an
initial substring, with the order according to $<^V$. For $m=0$, corresponding to the empty string, $E()= \Omega$.

We are interested in the first $k$ steps of the causal set process
specified by $\mu$; we need to consider all the \textit{likely} ways that
this process can begin. Accordingly, let $\cT$ be the set of strings
$y_1 \cdots y_j$ of elements of $V$, with $0\le j\le k$, such that
$\mu(E(y_1\cdots y_{i-1}y_i) \mid E(y_1\cdots y_{i-1})) > \delta$ for
$i=1, \ldots, j$. Note that the empty string is in $\cT$; also, by
definition, if a string is in $\cT$, then so is every initial substring
of it. Note also that $\mu(E(y_1\cdots y_j)) > \delta^j \ge\delta
^k$ for
every string $y_1\cdots y_j$ in~$\cT$. One consequence is that there are
at most $\delta^{-k}$ strings in $\cT$.

We enumerate the elements of $V$ as $v_1, v_2, \ldots.$ As the sum of the
probabilities $\mu( \{ \omega\dvtx v_j \in\Xi_q(\omega)\})$ is at most $q$,
there is some $m \in\N$ such that
\[
\sum_{j=m+1}^\infty\mu\bigl(\{ \omega\dvtx v_j \in\Xi_q(\omega)\}\bigr) <
\delta^k.
\]
Set $V_m = \{v_1, \ldots, v_m\}$. Note that every element appearing in a
string in $\cT$ is in $V_m$.

Now, for any string $y_1\cdots y_j$ of elements of $V_m$, we have
$\nu^n(E(y_1\cdots y_j))(\omega) \!\to\mu(E(y_1\cdots y_j))$ a.s., since
$\mu$ is essential. For $n \in\N$, we define
\begin{eqnarray*}
C_n &=& \{ \omega\dvtx |\nu^n(E(y_1\cdots y_j))(\omega)-\mu(E(y_1\cdots y_j))|
< \delta^k/3, \\
&&\mbox{ for all strings } y_1\cdots y_j
\mbox{ of at most $k$ distinct elements of $V_m$} \}.
\end{eqnarray*}
As there are only finitely many strings of at most $k$ distinct elements
of $V_m$, we have $\mu(C_n) \ge1-\eps/8$ for sufficiently large $n$.

%If $\omega\in C_n$, then we have,
%for all strings $y_1\cdots y_j$ in
%$\cT$,
%& = & \frac{\nu^n(E(y_1 \cdots y_j))(\omega)}
%{\nu^n(E(y_1\cdots y_{j-1}))(\omega)} \nonumber\\
%& > & \frac{\mu(E(y_1\cdots y_j)) - \delta^k/3}{\mu(E(y_1\cdots
%y_{j-1}))
%+ \delta^k/3} \nonumber\\
%& > & \frac{\frac{2}{3}\mu(E(y_1\cdots y_j))(\omega)}
%{\frac{4}{3}\mu(E(y_1\cdots y_{j-1}))(\omega)} >
If $\omega\in C_n$, $j < k$, $y_1\cdots y_j \in\cT$,
$y_1\cdots y_jy \notin\cT$, and $y \in V_m$, then
\begin{eqnarray}
\label{eq-cn0-b}
\nu^n(E(y_1\cdots y_jy) \mid E(y_1 \cdots y_j))(\omega)
& = & \frac{\nu^n(E(y_1 \cdots y_jy))(\omega)}
{\nu^n(E(y_1\cdots y_j))(\omega)} \nonumber\\
& < & \frac{\mu(E(y_1\cdots y_jy)) + \delta^k/3}{\mu(E(y_1\cdots y_j))
- \delta^k/3}\nonumber\\[-8pt]\\[-8pt]
& < & \frac{\delta\mu(E(y_1\cdots y_j))}{2/3\mu(E(y_1\cdots y_j))}
+ \frac{\delta^k/3}{\delta^{k-1} - \delta^{k-1}/3} \nonumber\\
& < & \frac{3}{2} \delta+ \frac{1}{2}\delta= 2\delta.\nonumber
\end{eqnarray}

We now fix some $n\ge q$ (depending on $k$, $\epsilon$, and also on the
measure $\mu$) large enough that $\mu(C_n) \ge1 -\eps/8$, and also
large enough for the conclusion of Lemma~\ref{lem:quant} to hold.

For $\omega\in\Omega$, we define a \textit{bad string} for $\omega$ to
be a
string $y_1\cdots y_jy$ of elements of $\Xi_n(\omega)$, with $j < k$,
such that $y_1\cdots y_j$ is in $\cT$, $y_1\cdots y_jy$ is not in $\cT$,
and $\nu^n(E(y_1\cdots y_jy) \mid E(y_1\cdots y_j))(\omega) \ge
2\delta$.
Set $F = \{ \omega\dvtx \mbox{ there are no bad strings }\break \mbox{for } \omega\}$.

We consider also the following events:
\begin{eqnarray*}
B^1 &=& \{ \omega\dvtx \Xi_q(\omega) \cap(V\sm V_m) = \varnothing\}, \\
B^2 &=& \{ \omega\dvtx \mbox{for all } z \in\Xi_q(\omega)\sm V,
\nu^n(G(z,k))(\omega) < \delta^{k+1} \}, \\
B^3 &=& \{ \omega\dvtx \mbox{for all } z \in
\Xi_n(\omega)\sm\Xi_q(\omega),
\nu^n(G(z,k))(\omega) < \delta^{k+1} \}.
\end{eqnarray*}
We claim that $C_n \cap B^1 \cap B^2 \cap B^3 \subseteq F$.

Indeed, if $\omega$ is in $C_n$, then from (\ref{eq-cn0-b}) it follows
that there are no bad strings $y_1\cdots y_jy$ for $\omega$ with
$y \in V_m$. If $\omega$ is in $B^1$, it certainly follows that there
are no bad strings with $y \in\Xi_q(\omega) \cap(V \sm V_m)$. Finally,
if $\omega$ is in $B^2 \cap B^3$, then
$\nu^n(G(z,k))(\omega) < \delta^{k+1}$ for all
$z \in\Xi_n(\omega) \sm(\Xi_q(\omega) \cap V)$; if $\omega$ is
also in
$C_n$, then this implies that
\[
\nu^n(E(y_1\cdots y_jz) \mid E(y_1\cdots y_j))(\omega) <
\frac{\delta^{k+1}}{\nu^n(E(y_1\cdots y_j))(\omega)} <
\frac{\delta^{k+1}}{\delta^k - \delta^k/3} < 2\delta
\]
for all strings $y_1\cdots y_j \in\cT$ and all
$z \in\Xi_n(\omega) \sm(\Xi_q(\omega) \cap V)$, and so there are
no bad
strings $y_1\cdots y_jz$ for $\omega$ with
$z \in\Xi_n(\omega) \sm(\Xi_q(\omega) \cap V)$.\vspace*{1pt} We conclude that, if
$\omega\in C_n \cap B^1 \cap B^2 \cap B^3$, then there are no bad strings
for $\omega$, and so $\omega\in F$.

We chose $n$ so that $\mu(C_n) \ge1 -\eps/8$, and so that
$\mu(B^2) \ge1-\eps/8$---see Lemma~\ref{lem:quant}. We also chose $m$
so that $\mu(B^1) \ge1 - \delta^k > 1-\eps/8$.

To estimate $\mu(B^3)$, we use Theorem~\ref{thm:DLR} to express this as
$\E_\mu(\nu^n(B^3))$. The event $B^3$ depends only on the finite
poset $\Pi_n(\omega)$, and $\nu^n(B^3)(\omega)$ is the probability
that, in a uniformly random linear extension of this poset, all the
elements $z$ with $\nu^n(G(z,k)) \ge\delta^{k+1}$ appear among the
first $q$.
Lemma~\ref{lem:q} now tells us that $\nu^n(B^3)(\omega) > 1 -\eps/8$
for every $\omega\in\Omega$, and therefore we have
$\mu(B^3) > 1-\eps/8$.

Hence, we have $\mu(F) \ge1-\eps/2$.

Now let
\[
H = \{ \omega' \in\Omega\dvtx \Xi_k(\omega') \subseteq
V \cup\cA(\omega') \}.
\]
%
%Again, we have $\mu(H) = \E_\mu(\nu^n(H))$, by Theorem~\ref{thm:DLR}.
For $\omega\in F$, we claim that $\nu^n(H)(\omega) \ge1-\eps/2$.

To verify the claim, we take $\omega\in F$, and apply
Lemma~\ref{lem:finite} to the finite poset $\Pi_n(\omega)$. Note that
$\nu^n(H)(\omega)$ is the probability that, in a uniformly random linear
extension of this poset, all the first $k$ elements are either in~$V$ or
minimal in $\Pi_n\sm V$---in other words the first $k$ elements do not
contain\vadjust{\eject} a pair of comparable elements that are not in $V$. We apply
Lemma~\ref{lem:finite} with~$\cZ$ equal to the family of sets
$\{z_1,\ldots,z_j\}$, for $z_1\cdots z_j$ a string in $\cT$ whose elements
are all in $\Pi_n(\omega)$, and with $\delta$ replaced by $2\delta
$. The
statement that $\omega\in F$ implies that the condition of
Lemma~\ref{lem:finite} on $\cZ$ is satisfied. As all the elements
appearing in strings in $\cT$ are in $V_m\subseteq V$, the union $Y$ of
the sets in $\cZ$ is a subset of $V$. We conclude from
Lemma~\ref{lem:finite} that
\[
\nu^n(H)(\omega) \ge1 - \pmatrix{k\cr2} 2 \delta =
1 - \pmatrix{k\cr 2} 2 \frac{\eps}{2k^2} \ge1 - \eps/2
\]
for all $\omega\in F$, as claimed.

We now have, by Theorem~\ref{thm:DLR} and our earlier calculations:
\[
\mu(H) = \E_\mu(\nu^n(H)) \ge\mu(F) (1-\eps/2) \ge(1-\eps/2)^2
\ge1-\eps.
\]
This establishes (\ref{eq:eps-k}), and completes the proof.
\end{pf*}

Now we have proved Theorem~\ref{thm:extremal}, it is possible to say more
about the nature of extremal order-invariant measures. In what follows,
we omit some of the details.

Suppose $Q=(Z,<)$ is a causal set or finite poset, with a set
$M \subseteq Z$ of marked maximal elements. Let $\mathcal{R}(Q)$ be the
set of causal sets obtained from~$Q$ by replacing each element $z$ of $M$
with a countably infinite antichain $A_z \subset[0,1]$. Let
$H_Q = \{ \omega\in\Omega\dvtx \Pi(\omega) \in\mathcal{R}(Q)\}$.
Theorem~\ref{thm:extremal} says that every extremal order-invariant
measure $\mu$ has $\mu(H_Q) =1$ for some $Q$. We may assume that the
down-sets $D(z)$, for $z \in M$, are all distinct: otherwise, we may
replace a set of elements of $M$ having the same down-set by a single
element of $M$. We say that an order-invariant measure $\mu$, or its
associated order-invariant process, \textit{generates} $Q$ if $\mu(H_Q)=1$.

Fix $Q$ and $M$ as above, and suppose $\mu$ is an extremal order-invariant
measure generating $Q$. If $M$ is empty, then $Q$ is a causal set, and
$\mu$ is a causet measure on the fixed causal set $Q$. Moreover, the set
of order-invariant measures on $Q$ is a convex subset of the set of all
order-invariant measures, and $\mu$ is an extremal element of this set.

If $M$ is non-empty, fix attention on one element $z\in M$, and let
$D = D(z)$ be the down-set of elements below $z$ in $Q$, all of which are
unmarked. For $\omega\in\Omega$ and $n \in\N$, let $R_n(\omega)$ be
the proportion of elements of $\Xi_n(\omega)$ that have down-set
equal to
$D$ in $\Pi_n(\omega)$. It is not too hard to show that there is some
$q>0$ such that, $\mu$-a.s., $R_n \to q$.

Now, for each $\omega\in H_Q$ such that $R_n(\omega) \to q$, let
$\zeta_1(\omega),\zeta_2(\omega), \ldots$ be the sequence of labels of
those elements of $\Xi(\omega)$ with $D(\zeta_i) = D$. Order-invariance
implies that the sequence $(\zeta_1,\zeta_2, \ldots)$ is a sequence of
exchangeable random variables. Therefore, by the Hewitt--Savage theorem,
as $\mu$ is extremal, there is some probability distribution $\rho$ on
$[0,1]$ such that the $\zeta_i$ are i.i.d. random variables with distribution
$\rho$.

Given a number $q \in(0,1]$, a probability distribution $\rho$ on
$[0,1]$, and a~marked element $z \in M$, we say that an order-invariant
causet process generating $Q$ \textit{produces $z$ with parameters
$(q,\rho)$} if, at every stage after all elements of $D(z)$ have been
generated, $<^{[k+1]}$ is obtained from $<^{[k]}$ by placing $k+1$ above
the elements in the set of indices corresponding to $D(z)$ with
probability $q$, and---conditioned on that event---the new element
$\xi_{k+1}$ has distribution~$\rho$.

Suppose $Q=(Z,<)$ is a causal set or finite poset with a marked set~$M$ of
maximal elements, and $\mu$ is an extremal order-invariant measure that
generates $Q$. It can be shown that, for each element $z$ of $M$, there
is a~number $q(z) \in(0,1]$ and a probability distribution $\rho(z)$ on
$[0,1]$ such that $\mu$ produces~$z$ with parameters $(q(z),\rho(z))$.

The sum of the $q(z)$ must be at most~1, as these are probabilities of
disjoint events; if the poset $Q$ is finite, then the sum must equal~1.

If there is an extremal order-invariant measure generating $Q$, producing
each $z \in M$ with parameters $(q(z),\rho(z))$, then for any other
set of
distributions $(\rho'(z))$, there is also an extremal order-invariant
measure generating~$Q$, producing each $z \in M$ with parameters
$(q(z),\rho'(z))$. To obtain this, we simply change the description of
the order-invariant process, replacing each~$\rho(z)$ by $\rho'(z)$.

Moreover, given a causal set $Q$ with a set $M$ of marked maximal
elements, define $Q'$ by replacing each element $z$ of $M$ with an
infinite chain $C_z$, labeled arbitrarily in such a way that all labels
are distinct. Given an extremal order-invariant process generating $Q$,
producing each $z \in M$ with parameters $(q(z),\rho(z))$, then we can
obtain an extremal order-invariant process on the fixed causal set $Q'$:
whenever the original process calls for an element with down-set $D(z)$,
in the new process we take the next element of the chain $C_z$.

This process is reversible: if we have an extremal order-invariant measure
on some fixed causal set $P=(X,<)$, and there is an infinite chain $C$
in~$P$ such that all elements of $C$ are above some set $D$ and incomparable
to all elements of $X \sm(C \cup D)$, then we obtain an order-invariant
measure generating the poset $Q$ obtained from $P$ by replacing $C$ by a
single marked maximal element.

Therefore, in order to describe all extremal order-invariant measures, it
suffices to describe all extremal order-invariant measures on fixed causal
sets~$P$. This serves to motivate the work in our companion
paper~\cite{BL2}.

% imsref loaded by dianan, 2010-12-15 13:35:27
%

\printaddresses

\end{document}